\documentclass[reqno]{amsart}
\usepackage{amssymb,amsthm,amsfonts,amstext}
\usepackage{amsmath}
\usepackage{graphicx}
\usepackage{eucal}
\usepackage[]{float}
\usepackage[]{epsfig}
\usepackage[]{pstricks}

\usepackage{graphicx}
\usepackage{amssymb,amstext}
\usepackage{amsmath}
\usepackage{url}
\usepackage{hyperref}
\usepackage{bbm}
\DeclareMathOperator{\supp}{supp}

\DeclareMathOperator{\Her}{Her}
\DeclareMathOperator{\her}{her}
\renewcommand\thefigure{\thesection.\@arabic\c@figure}
\renewcommand\thetable{\thesection.\@arabic\c@table}
\newcommand{\mc}[1]{{\mathcal #1}}

\newcommand{\bb}[1]{{\mathbb #1}}
\makeatletter \@addtoreset{equation}{section} \makeatother

\renewcommand\thetable{\thesection.\@arabic\c@table}
\newtheorem{theorem}{Theorem}[section]
\newtheorem{lemma}[theorem]{Lemma}
\newtheorem{proposition}[theorem]{Proposition}
\newtheorem{corollary}[theorem]{Corollary}

\newtheorem{definition}{Definition}
\newtheorem{remark}{Remark}[section]

\newcommand{\<}{\langle}
\renewcommand{\>}{\rangle}
\begin{document}

\title{Nonlinear fluctuations of weakly asymmetric interacting particle systems}

\author{Patr\'{\i}cia Gon\c{c}alves}
\address{CMAT, Centro de Matem\'atica da Universidade do Minho, Campus de Gualtar, 4710-057 Braga, Portugal\\Tel.: +351-253-604080\\
              Fax: +351-253-604369\\}
\email{patg@math.uminho.pt\\patg@impa.br}

\author{Milton Jara}
\address{IMPA, Instituto Nacional de Matem\'atica Pura e Aplicada,
Estrada Dona Castorina 110,
Jardim Bot\^anico,
22460-320 Rio de Janeiro,
Brazil and Ceremade, UMR CNRS 7534,
    Universit\'e Paris-Dauphine,
    Place du Mar\'echal De Lattre De Tassigny,
    75775 Paris Cedex 16,  France\\
 Tel.: +55-21-25295230\\
              Fax: +55-21-25124115\\}
\email {mjara@impa.br}
\date{\today}
\renewcommand{\subjclassname}{\textup{2000} Mathematics Subject Classification}


\begin{abstract}
We introduce what we call the second-order Boltzmann-Gibbs principle, which allows to replace local functionals of a conservative, one-dimensional stochastic process by a possibly nonlinear function of the conserved quantity. This replacement opens the way to obtain nonlinear stochastic evolutions as the limit of the fluctuations of the conserved quantity around stationary states. As an application of this second-order Boltzmann-Gibbs principle, we introduce the notion of energy solutions of the KPZ and stochastic Burgers equations. Under minimal assumptions, we prove that the density fluctuations of one-dimensional, stationary, weakly asymmetric, conservative particle systems are sequentially compact and that any limit point is given by energy solutions of the stochastic Burgers equation. We also show that the fluctuations of the height function associated to these models are given by energy solutions of the KPZ equation in this sense. Unfortunately, we lack a uniqueness result for these energy solutions. We conjecture these solutions to be unique, and we show some regularity results for energy solutions of the KPZ/Burgers equation, supporting this conjecture.
\end{abstract}

\maketitle

\section{Introduction}
In a seminal article \cite{KPZ}, Kardar, Parisi and Zhang proposed a phenomenological model for  the stochastic evolution of the profile of a growing interface $h_t(x)$. The so-called KPZ equation has the following form in one dimension:
\begin{equation*}
d h_t = \nu \Delta h_t dt + \lambda \big(\nabla h_t\big)^2 dt+ \sigma  {\mc W}_t,
\end{equation*}
where $\mc W_t$ is a space-time white noise and the constants $\nu$, $\lambda$, $\sigma$ are related to some thermodynamic properties of the interface. The quantity $h_t(x)$ represents the {\em height} of the interface at the point $x \in \bb R$ and at time $t\in{\mathbb{R}_+}$. From a mathematical point of view, this equation is ill-posed, since the solutions are expected to look locally like a Brownian motion, and in this case the nonlinear term does not make sense, at least not in a classical sense.
This equation can be solved at a formal level using the Cole-Hopf transformation $z(t,x) = \exp\{\gamma h_t(x)\}$ for $\gamma = \frac{\lambda}{\nu}$, which transforms this equation into the stochastic heat equation $d z_t = \nu \Delta z_t dt+ \frac{\lambda \sigma}{\nu}z_t {\mc W}_t$. This equation is now linear, and mild solutions can be easily constructed. Fairly general uniqueness and existence criteria can be obtained as well. We will call these solutions {\em Cole-Hopf} solutions. It is widely believed that the physically relevant solutions of the KPZ equation are the Cole-Hopf solutions. However, the KPZ equation has been so resistant to any attempt to mathematical rigor, that up to now it has not even been proved that Cole-Hopf solutions satisfy the KPZ equation in any meaningful sense. Some interpretations that allow rigorous results are proved to give non-physical solutions \cite{Chan}. Up to our knowledge, the best effort in this direction corresponds to the work of \cite{BG}. In that article, the authors prove two results. First, they prove that  Cole-Hopf solutions can be obtained as the limit of a sequence of mollified versions of the KPZ equation. Second, they proved that the Cole-Hopf solution appears as the scaling limit of the fluctuations of the current for the weakly asymmetric simple exclusion process (WASEP),\footnote{The weakly asymmetric simple exclusion process is not a process, but a {\em family} of processes on which the asymmetry is scaled down with respect to a scale parameter $n$, cf. Section \ref{s1.2}. The denomination `weakly asymmetric process' for this family of processes, although not completely correct, is widely used in the literature.} giving mathematical support to the physical relevance of the Cole-Hopf solution.

The stochastic Burgers equation with conservative noise corresponds to the equation satisfied by the slope of the height function. Define $\mc Y_t =  \nabla h_t$. Then, $\mc Y_t$ satisfies the equation
$d \mathcal Y_t = \nu \Delta \mathcal Y_t dt + \lambda \nabla \mc Y_t^2 dt+ \sigma \nabla {\mc W}_t$. This equation has (always at a formal level!) a spatial white noise as an invariant solution. In this case it is even clearer that some procedure is needed in order to define $\mc Y_t^2$ in a proper way.

Since the groundbreaking works of \cite{Joh} and \cite{BDJ}, a new approach to the analysis of what is called the {\em KPZ universality class} has emerged. The general strategy is to describe various functionals of one-dimensional asymmetric, conservative systems in terms of determinantal formulas. These determinantal formulas turn out to be related to different scaling limits appearing in random matrix theory, from where scaling limits of those functionals can be obtained. We refer to the expository work \cite{FS} for further references and more detailed comments. The most studied model that fits into the setting described above is the totally asymmetric simple exclusion process. A second breakthrough was the generalization of these formulas \cite{TW} to the asymmetric simple exclusion process. These formulas opened a way to obtain an explicit description of the solutions of the KPZ equation for the so-called {\em wedge initial condition} \cite{ACQ,SS}, built over the work of \cite{BG}. In particular, these works provide strong evidence that the KPZ equation serves as a {\em crossover equation}, connecting the (nowadays well-understood, see \cite{Spo}) Edwards-Wilkinson universality class to the KPZ universality class. In \cite{Spo2}, the author coined the term {\em stochastic integrability} to emphasize the dependence of the methods in the somehow exact solvability of the models considered. Another approach to the analysis of fluctuations of one-dimensional conservative systems was proposed in \cite{BS}. The authors call their approach {\em microscopic concavity/convexity}  and it is exploited in \cite{BQS} in order to prove that the Cole-Hopf solution of the KPZ equation has the scaling exponents predicted by physicists. Although this approach only gives information about scaling exponents and not about distributional limits, it does not rely on the stochastic integrability of the models considered.

The main drawback of all these approaches is the lack of robustness. The microscopic Cole-Hopf transformation used in \cite{BG} takes advantage of special combinatorial features for the WASEP, which are destroyed by any other interaction different from the exclusion principle. Determinantal processes appear in a natural way for particle systems that can be described by conditioning non-interacting systems to non-intersection. The microscopic concavity/convexity property can be defined only for attractive systems, and it has been proved only for the asymmetric simple exclusion and the totally asymmetric nearest-neighbors zero-range, under very restrictive assumptions. Moreover, up to our knowledge all the approaches to the KPZ equation go through the construction of  Cole-Hopf solutions as initially proposed in \cite{BG} (see however \cite{QV}).

It is widely believed in the physics community that the KPZ equation governs the large-scale properties of one-dimensional asymmetric conservative systems in great generality. The microscopic details of each model should only appear through the values of the constants $\nu$, $\lambda$ and $\sigma$. In this article we provide a new approach which is robust enough to apply for a wide family of one-dimensional asymmetric systems. The payback of such a general approach comes at the level of the results: we can not prove the precise results of \cite{BG}, and we can not recover the detailed results obtained by the stochastic integrability approach. However, our approach is robust enough to give information about sample path properties of the solutions of the KPZ/Burgers equation. As a stochastic partial differential equation, the main problem with the stochastic Burgers equation is the definition of the square $\mc Y_t^2$. Spatial white noise is expected to be its only invariant solution, and it is expected that physically relevant solutions look locally like white noise. But if $\mc Y_t$ has the distribution of a white noise, $\mc Y_t^2=+\infty$ $a.s.$, so we need a way to deal with the singular term $\mc Y_t^2$.

In order to deal with this singular term, we introduce a new mathematical tool, which we call {\em second-order Boltzmann-Gibbs principle}. The usual Boltzmann-Gibbs principle, introduced in \cite{BR} and proved in \cite{D-MPSW} in our context, basically states that the space-time fluctuations of any field associated to a conservative model can be written as a linear functional of the density field $\mc Y_t^n$. Our second-order Boltzmann-Gibbs principle states that the first-order correction of this limit is given by a singular, quadratic functional of the density field. It has been proved that in dimension $d \geq 3$, this first order correction is given by a white noise \cite{CLO}. It is conjectured that this is also the case in dimension $d=2$ and in dimension $d=1$ if our first-order correction is null. The proof of this second-order Boltzmann-Gibbs principle relies on a multiscale analysis which is very reminiscent of the one-block, two-blocks scheme introduced in \cite{GPV}. In a very different context, a somehow similar multi scale scheme was introduced in \cite{EMY} and \cite{Gon}. A key point in our second-order Boltzmann-Gibbs principle is that we are able to obtain sharp quantitative bounds on the error performed in the aforementioned replacement. In particular, these quantitative bounds are of the exact order needed to make sense of the singular term $\mc Y_t^2$. Our proof basically relies in two main ingredients: a sharp uniform lower bound on the inverse of the relaxation time of the dynamics restricted to a finite interval (the so-called {\em spectral gap inequality}), and a sharp estimate on the expectation of a local function with respect to a stationary measure, conditioned to the value of the conserved quantity on a big interval (the so-called {\em equivalence of ensembles}). In particular, our proof is robust and can be applied for a wide range of conservative, one-dimensional stochastic processes.

In this article, we focus on the applications of the aforementioned second-order Boltzmann-Gibbs principle to the analysis of weakly asymmetric, conservative particle systems and their connections with the KPZ/Burgers equation. It is possible to use the second-order Boltzmann-Gibbs principle introduced in this article to obtain various results in related contexts.
In \cite{GJ} we use it in order to obtain scaling limits of some observables of one-dimensional systems, like the occupation time of the origin. In particular, we are able to show the existence of the process $\int_0^t \nabla h_s(0)ds$ and to show the convergence of the occupation time of the origin to this process for the models considered in this article. In \cite{GJ2} we show the irrelevance of the asymmetry for scalings weaker that the one considered here. And in \cite{GJ3} it is shown the existence of solutions of superdiffusive stochastic Burgers equations, and the uniqueness of solutions of hyperviscous, stochastic Burgers equations, improving previous results in the literature.

Going back to the context of the KPZ/Burgers equation, our first contribution is the introduction of the notion of (stationary) {\em energy solutions} of the KPZ/Burgers equation (see Sections \ref{s1.5} and \ref{s1.6}). Various attempts to define a solution of the Burgers equation in a rigorous way have been made. One possibility is to regularize the noise $\mc W_t$ and then to turn the regularization off. For the regularized problem, several properties, like well-posedness of the Cauchy problem and existence of invariant measures can be proved (\cite{Sin} and \cite{EKMS}). However, the available results hold in a window which is still far from the white noise $\mc W_t$. Another possibility is to regularize the nonlinearity $\nabla \mc Y_t^2$ \cite{D-PD}. Once more, this procedure gives well-posedness in a window which is far from the Burgers equation. Another possibility corresponds to define the nonlinear term through a sort of Wick renormalization \cite{HLOUZ}. However, this procedure does not lead to physical solutions \cite{Chan}. Our notion of energy solutions is strong enough to imply some regularity properties of the solutions which allow to justify some formal manipulations. The idea behind this notion of solutions is that the nonlinear term gets regularized by the noise {\em in the time variable}. This regularization makes possible to obtain bounds on the space-time variance of the nonlinear term while the spatial variance of the nonlinear term explodes.

We introduce the notion of energy solutions of the KPZ equation (and its companion stochastic Burgers equation) in order to state in a rigorous way our second contribution, which we describe as follows. Take a one-dimensional, weakly asymmetric conservative particle system and consider the rescaled space-time fluctuations of the density field $\mc Y_t^n$ (see Section \ref{s1.7}). The strength of the asymmetry is of order $1/\sqrt n$. For the speed-change simple exclusion process considered in \cite{FHU} and starting from a stationary distribution, we prove tightness of the sequence of processes $\{\mc Y_t^n\}_{n\in{\bb N}}$ and we prove that any limit point of $\{\mc Y_t^n\}_{n\in{\bb N}}$ is an energy solution of the stochastic Burgers equation. As we mentioned above, the only ingredients needed in order to prove this result are a sharp estimate on the spectral gap of the dynamics restricted to finite boxes (see Proposition \ref{p2}) and a strong form of the equivalence of ensembles for the stationary distribution (see Proposition \ref{p6}). Therefore, our approach works, modulo technical modifications, for any one-dimensional, weakly asymmetric conservative particle system satisfying these two properties (see \cite{GJS}). In particular, our approach is suitable to treat models like the zero-range process and the Ginzburg-Landau model in dimension $d=1$, for which the methods mentioned before fail dramatically. Our approach also works for models with finite-range, non-nearest neighbor interactions with basically notational modifications.

If we could show that energy solutions of the KPZ equation are unique, we would be able to obtain a full convergence result for the fluctuation field $\mc Y_t^n$. This article is a revised version of \cite{GJ0}. At the time of publication of \cite{GJ0}, there was no rigorous, direct  definition of solutions of the KPZ equation. During the preparation of this revised version, the article \cite{Hai} has finally achieved the goal of giving a rigorous sense to the KPZ equation.
At the light of the results of \cite{Hai}, we feel inclined to say that the energy condition is not enough to guarantee uniqueness of the KPZ equation, but an argument proving or disproving this is not yet available.

We are also able to show the analogous results, namely tightness and convergence along subsequences to energy solutions of the KPZ equation, for the height fluctuation field associated to $\mc Y_t^n$. The height fluctuation field $\mc H_t^n$ is formally obtained as $\mc Y_t^n = - \nabla \mc H_t^n$. However, an integration constant is missing. We will see that for energy solutions of the stochastic Burgers equation, there is a canonical way to recover this integration constant, which matches the one obtained looking at the microscopic models. We point out as well that our approach works in finite or infinite volume with few modifications (to this respect, we notice that the extensions of the results in  \cite{ACQ,SS} to finite volume require an extra argument, as well as the extensions of the results in \cite{Hai} to infinite volume).

We consider in this article speed-change exclusion processes satisfying the so-called {\em  gradient condition}. Notice that we need to know the invariant measures of the model in order to state the equivalence of ensembles.  A simple way to verify that the invariance of a measure under the symmetric dynamics is preserved by introducing an asymmetry, is to check if the model satisfies the gradient condition. It is only at this point that we need the gradient condition. In particular, our approach also works for weakly asymmetric systems for which the invariant measures are known explicitly, even if the gradient condition is not satisfied.
In view of this discussion, we say that energy solutions of the KPZ/Burgers equation are {\em universal}, in the sense that they arise as the scaling limit of the density in one-dimensional, weakly asymmetric conservative systems satisfying fairly general assumptions.

Another consequence of our results is an {\em Einstein relation}, that is, an exact description of the constants appearing in the KPZ equation in terms of thermodynamic quantities associated to the model in consideration. Let us denote by $\rho$ the density of particles. Then, $\nu = D(\rho)$ turns out to be the diffusivity of the symmetric (or unperturbed) model. The variance of the noise is $\sigma^2 = 2 \chi(\rho) D(\rho)$, where $\chi(\rho)$ is the static compressibility of the model. And finally, we have a second-order Einstein relation for the nonlinear term: $\lambda = \frac{1}{2} H''(\rho)$, where $H(\rho) = a \chi(\rho) D(\rho)$ is the flux associated to the asymmetric model and $a$ is the strength of the weak asymmetry. We point out that for stochastically integrable models, these constants are independent of the density $\rho$.

This article is organized as follows. In Section \ref{s1} we give precise definitions of the model considered here and we state the results proved in the rest of the article. In Section \ref{s1.1} we define the speed-change exclusion process and we state some of its basic properties. In Section \ref{s1.2} we define the weakly asymmetric simple exclusion process with speed change. In Sections \ref{s1.3}-\ref{s1.6} we introduce the functional spaces on which the fluctuation fields will live, and we introduce the notion of energy solutions of KPZ/Burgers equation. In Sections \ref{s1.7} and \ref{s1.8} we define the density, height and current fluctuation fields and we state our main results.

In Section \ref{s2} we introduce general tools from the theory of Markov processes and from statistical mechanics which are the main ingredients of the proof of the second-order Boltzmann-Gibbs principle. In particular, whenever we can prove Corollary \ref{c1} and Proposition \ref{p6}, aside some technical assumptions, the proofs in Section \ref{s4} remain true.

In Section \ref{s3} we state and prove the second-order Boltzmann-Gibbs principle, which is the main technical innovation of this article.  In Section \ref{s3.1} we state the classical Boltzmann-Gibbs principle, first introduced in \cite{BR}. In Section \ref{s3.2} we introduce and prove the second-order Boltzmann-Gibbs principle, by means of a multiscale analysis introduced in \cite{Gon} (see also \cite{EMY}). The multiscale analysis is very reminiscent of the {\em one-block} and the {\em two-blocks} scheme introduced in \cite{GPV}. The first step, which we call {\em one-block estimate} (see Lemma \ref{l2}) by analogy with \cite{GPV}, allows to replace a local function by a function of the density of particles in a box of finite size. A second step, which we call the {\em renormalization step} (see Lemma \ref{l3}) allows to double the size of the box obtained in the one-block estimate. Using this Lemma a
number of times which is logarithmic on the size of the box, we prove the
{\em two-blocks estimate} (see Lemma \ref{l4}), which allows to replace a function of the density in a box of finite size by a function of the density over a box which will be taken macroscopically small in Section \ref{s4}. In Lemma \ref{l5} we replace this function of the density of particles by a quadratic function.

In Section \ref{s4} we prove our main result for the fluctuations of the density of particles. The proof follows the classical scheme to prove convergence theorems in probability. In Section \ref{s4.1}, we introduce some martingales which will be very useful. In Section \ref{s4.2} we prove tightness of the density fluctuation field and in Section \ref{s4.3} we prove that any limit point is a stationary energy solution of the stochastic Burgers equation. We conjecture that energy solutions starting from the stationary state are unique in distribution. Conditioned to this uniqueness result, convergence would follow.

In Section \ref{s6} we prove various path properties of energy solutions of the stochastic Burgers equation. The continuity properties stated in Theorem \ref{t2.1.5} are close to optimal, as we can see by comparing our results with the ones in \cite{Hai}. In particular, we prove that the stationary Cole-Hopf solution constructed in \cite{BG} is also an energy solution of the stochastic Burgers equation. Notice that this result, combined with Theorem 6.6 of \cite{GJ}, shows that the process $\int_0^t \nabla h_s(0)ds$ is well-defined for the Cole-Hopf solution of the KPZ equation, a new result which is not accessible using previous methods.

In Section \ref{s5}, we finish the article by proving Theorem \ref{t2.1.6} for the height fluctuation field. The height fluctuation field formally corresponds to the integral of the density fluctuation field. The convergence of the height fluctuation field does not follow directly from the convergence of the density, since we need to deal with a constant of integration, which is a non-trivial process that we relate to a sort of mollified current process.

\begin{remark}
Throughout the article, we use the denomination ``Proposition'' for results which are proved elsewhere and the denomination ``Theorem'' (or ``Lemma'') for original main (auxiliary) results.
\end{remark}

\begin{remark}
This article is very detailed; some definitions are given more than once and some proofs may look rather standard for some readers. We believe the advanced reader will not have any difficulty skipping some material. In particular, at the cost of repetition, we have tried to keep the material on interacting particle systems and the material on stochastic differential equations as independent as possible.
\end{remark}

\section{Notation and results}
\label{s1}
\subsection{Speed-change exclusion process}
\label{s1.1}

Let $\Omega = \{0,1\}^{\bb Z}$ be the state space of a Markov process to be defined below. We denote the elements of $\Omega$ by $\eta = \{\eta(x); x \in \bb Z\}$. A function $f: \Omega \to \bb R$ is said to be {\em local} if there exists a finite set $A \subseteq \bb Z$ such that $f(\eta) = f(\xi)$ for any $\eta, \xi \in \Omega$ such that $\eta(x) = \xi(x)$ for any $x \in A$. We denote by $\supp(f)$ the smallest of such sets $A \subseteq \bb Z$, and we call $\supp(f)$ the {\em support} of the local function $f$. Let $r: \Omega \to \bb R$ be a positive function satisfying the following conditions:

\begin{itemize}
\item[i)] {\em Ellipticity.} There exists $\epsilon_0>0$ such that $\epsilon_0 \leq r(\eta) \leq \epsilon_0^{-1}$ for any $\eta \in \Omega$.
\item[ii)] {\em Finite range.} The function $r$ is local.
\item[ii)]{\em Reversibility.} For any $\eta, \xi \in \Omega$ such that $\eta(x) = \xi(x)$ for every $x \neq 0,1$, $r(\eta) = r(\xi)$.
\end{itemize}

Let $x \in \bb Z$ and let $\tau_x: \Omega \to \Omega$ be the translation in $x$: $\tau_x \eta(z) = \eta(z+x)$ for every $\eta \in \Omega$ and every $z \in \bb Z$. Let $f:\Omega \to \bb R$ be a given function. We denote by $\tau_xf :\Omega \to \bb R$ the function defined by $\tau_xf(\eta) = f(\tau_x \eta)$.

Define $r_x = \tau_x r$. The simple exclusion process with speed change $r$ is defined as the Markov process $\{\eta_t; t \geq 0\}$ with state space $\Omega$ and generated by the operator $S$ given by
\[
Sf(\eta) = \sum_{x \in \bb Z} r_x(\eta) \nabla_{x,x+1} f(\eta)
\]
for any local function $f: \Omega \to \bb R$. In this identity, $\nabla_{x,x+1} f(\eta)$ is defined as $f(\eta^{x,x+1})-f(\eta)$ and $\eta^{x,x+1} \in \Omega$ is defined as
\[
\eta^{x,x+1}(z)=
\left\{
\begin{array}{rl}
\eta(x+1);& z=x\\
\eta(x);& z=x+1\\
\eta(z);& z\neq x,x+1.\\
\end{array}
\right.
\]

Thanks to conditions i) and ii), Theorem I.3.9 of \cite{Lig} guarantees the existence of the process $\{\eta_t; t \geq 0\}$. The dynamics of this process can be described in the following way.
We call the elements of $\Omega$ {\em configurations}. We call the elements $x$ of $\bb Z$ {\em sites}. We say that, according to a configuration $\eta \in \Omega$, there is a particle at site $x \in \bb Z$ if $\eta(x)=1$. In this case we also say that the site $x$ is {\em occupied} by a particle. If $\eta(x)=0$, we say that the site $x$ is empty. A particle at site $x \in \bb Z$ waits an exponential time of instantaneous rate $r_x(\eta_t)$, at the end of which it tries to jump to the site $x+1$. If the site $x+1$ is empty, the jump is accomplished. If the site $x+1$ is occupied, the particle stays at site $x$. Simultaneously, this particle waits an exponential time of instantaneous rate $r_{x-1}(\eta_t)$, independent of the previous time, at the end of which the particle tries to jump to the site $x-1$ according to the exclusion rule detailed above. Each particle follows this dynamics, in a way conditionally independent with respect to $\sigma(\eta_s;s \leq t)$.

Let $\rho \in [0,1]$ and let $\nu_\rho$ be the Bernoulli product  measure in $\Omega$ of density $\rho$, that is, $\nu_\rho$ is the only probability measure in $\Omega$ such that
\[
\nu_\rho\{\eta\in\Omega:\,\eta(x_i)=1 \text{ for } i=1,...,\ell\} = \rho^\ell
\]
for any finite set $\{x_1,...,x_\ell\} \subseteq \bb Z$. The family of measures $\{\nu_\rho; \rho \in [0,1]\}$ is invariant and reversible with respect to the evolution of $\{\eta_t; t \geq 0\}$, thanks to condition iii). Condition i) implies that the family $\{\nu_\rho; \rho \in [0,1]\}$ is also ergodic with respect to the evolution of $\eta_t$.

\subsection{Weakly asymmetric process}
\label{s1.2}
In this section we introduce a weak asymmetry in the model described in Section \ref{s1.1}. Let $a \in \bb R$ and let $n \in \bb N$ be a scaling parameter. For $\eta\in{\Omega}$,  let us define
\[
  r_x^n(\eta) = r_x(\eta) \big\{ 1 + \frac{a}{\sqrt n} \eta(x+1)\big(1-\eta(x)\big)\big\}
\]
and let us define the operator $L_n$ as
\[
L_n f(\eta) = \sum_{x \in \bb Z} r_x^n(\eta) \nabla_{x,x+1} f(\eta)
\]
for any local function $f: \Omega \to \bb R$. If $a \geq 0$ or if $a<0$ and $n \geq a^2$, the operator $L_n$ turns out to be the generator of a Markov process in $\Omega$. Since $n$ is a scaling parameter which at some point will go to $\infty$, we will always assume that $L_n$ is a generator of a Markov process in $\Omega$. The dynamics of this process is easy to understand in terms of the dynamics of the process generated by the operator $S$. For $a>0$, jumps to the left, whenever possible, happen with an additional rate $an^{-1/2} r_x(\eta)$. For $a<0$, each jump to the left is suppressed with probability $-an^{-1/2}$ (the restriction $n \geq a^2$ is needed to have a well-defined suppression probability). Let $\{\eta_t^n; t \geq 0\}$ be the Markov process generated by the operator $n^2L_n$. The prefactor $n^2$ introduces an $n^2$ acceleration into the definition of the process $\{\eta_t^n; t \geq 0\}$ with respect to $\{\eta_t; t \geq 0\}$. We call the process $\{\eta_t^n; t \geq 0\}$ the {\em weakly asymmetric}\footnote{
The introduction of the asymmetry $a$ can also be understood as the introduction of a weak field of intensity $a$ (see \cite{Cha}).
}, simple exclusion process with speed change $r$
and accelerated by $n^2$.

It turns out that the measures $\{\nu_\rho; \rho \in [0,1]\}$ are still invariant with respect to the modified dynamics $\{\eta_t^n; t \geq 0\}$ {\em if and only if} the function $r$ satisfies the so-called {\em gradient condition}, which we enounce as follows:

\begin{itemize}
\item[iv)]{\em Gradient condition}. There exists a local function $\omega: \Omega \to \bb R$ such that
\[
r(\eta)\big(\eta(1)-\eta(0)\big) = \tau_1 \omega(\eta)- \omega(\eta)
\]
for any $\eta \in \Omega$.
\end{itemize}

We point out that in the case in which the gradient condition iv) is not satisfied, we do not know basically anything about the invariant measures of the process $\{\eta_t^n; t \geq 0\}$ (see \cite{Sep} for example). In this article, we only need to assume condition iv) in order to have a reasonable description of the invariant measures of the process $\{\eta_t^n; t \geq 0\}$.

From now on and up to the end of this article, we fix a density $\rho \in (0,1)$ and we consider the process $\{\eta_t^n; t \geq 0\}$ with initial distribution $\nu_\rho$. In particular, for any $t \geq 0$, the distribution of $\eta_t^n$ is $\nu_\rho$. We say in that case that the process $\{\eta_t^n; t \geq 0\}$ is {\em stationary}. We fix $T>0$, we call $\bb P_n$ the distribution of the process $\{\eta_t^n; t \in [0,T]\}$ in the space of c\`adl\`ag trajectories $\mc D([0,T];\Omega)$ and we call $\bb E_n$ the expectation with respect to $\bb P_n$.

\subsection{White noise as a random distribution}
\label{s1.3}
In this section we introduce various notions and definitions from functional analysis which will be needed later when we talk about scaling limits of the density of particles for the models described above. Let us denote by $\<\cdot,\cdot\>$ the inner product in $L^2(\bb R)$: $\<u,v\> = \int_{\mathbb R} u(x)v(x) dx$ for any $u,v \in L^2(\bb R)$. Let $u: \bb R \to \bb R$ be a function belonging to the space $\mc C^\infty(\bb R)$ of infinitely differentiable functions. For each $k \in \bb N_0$
\footnote{Here and below we use the notations $\bb N =\{1,2,...\}$ and $\bb N_0 = \{0,1,2...\}$},
define
\[
\|u\|_k = \<u,(-\Delta+\tfrac{x^2}{4})^k u\>^{1/2},
\]
where $\Delta$ is the usual Laplacian operator. The {\em Sobolev space} $H_k(\bb R)$ of order $k$ is defined as the closure of the set $\{u \in \mc C^\infty(\bb R); \|u\|_k<+\infty\}$ with respect to the norm $\|\cdot\|_k$, and it turns out to be a Hilbert space of inner product formally given by $\<u,(-\Delta+\tfrac{x^2}{4})^k v\>$ (this formula being correct if $v$ is regular enough). Notice that $H_0(\bb R) = L^2(\bb R)$.
The Schwartz space of test functions is defined as
\[
\mc S(\bb R) = \bigcap_{k \in \bb N_0} H_k(\bb R).
\]
A simple computation shows that $u \in \mc S(\bb R)$ if and only if $u \in \mc C^\infty(\bb R)$ and moreover $\sup_x |x|^m|u^{(\ell)}(x)|<+\infty$ for any $\ell, m \in \bb N_0$. The space $\mc S(\bb R)$ is complete and separable with respect to the metric $d: \mc S(\bb R) \times \mc S(\bb R) \to [0,1]$ given by
\[
d(u,v) = \sum_{k \in \bb N_0} \frac{1}{2^k} \min\{\|u-v\|_k,1\}.
\]
Let $\mc S'(\bb R)$ be the topological dual of $\mc S(\bb R)$. The space $\mc S'(\bb R)$ is known as the space of {\em tempered distributions} on $\bb R$. It can be shown that
\[
\mc S'(\bb R) = \bigcup_{k \in \bb N_0} H_{-k}(\bb R),
\]
where $H_{-k}(\bb R)$ is the topological dual of $H_k(\bb R)$. The strong topology on $\mc S'(\bb R)$ can be defined as follows. We say that a sequence $\{y_n\}_{ n \in \bb N}$ converges strongly to $y \in \mc S'(\bb R)$ if
\[
\lim_{n \to \infty} \sup_{u \in K} |y_n(u)-y(u)| =0
\]
for any compact set $K \subseteq \mc S(\bb R)$.
It turns out that $\mc S'(\bb R)$ is metrizable, separable and complete with respect to the strong topology. We say that a sequence $\{y_n\}_{ n \in \bb N}$ converges weakly to $y \in \mc S'(\bb R)$ if
\[
\lim_{n \to \infty} y_n(u) = y(u)
\]
for any $u \in \mc S(\bb R)$. It turns out that weak and strong convergence are equivalent in $\mc S'(\bb R)$ (this is not true for weak and strong {\em topologies}).
For more details and further properties of the space of tempered distributions, see \cite{Tre}.

For each $\ell \in \bb N_0$, define the {\em Hermite polynomial}
\footnote{These polynomials are known in the literature as the {\em probabilistic} Hermite polynomials.}
 of index $\ell$ as the function $\Her_\ell: \bb R \to \bb R$ given by
\[
\Her_\ell(x) = (-1)^\ell e^{\frac{x^2}{2}} \frac{d^\ell}{dx^\ell} e^{-\frac{x^2}{2}},
\]
and let $\her_\ell (x) = \sqrt{\ell !\sqrt{2 \pi}} \Her_\ell(x) e^{-\frac{x^2}{4}}$ be the Hermite function of index $\ell$. Notice that $\her_\ell \in \mc S(\bb R)$ for any $\ell \in \bb N_0$. The vector space generated by $\{\her_\ell\}_{\ell \in \bb N_0}$ is dense in $\mc S(\bb R)$. A (not so) straightforward computation shows that
\[
(-\Delta+ \tfrac{x^2}{4}) \her_\ell(x) = (\ell+\tfrac{1}{2}) \her_\ell(x)
\]
and in particular $\{\her_\ell\}_{\ell \in \bb N_0}$ is an orthogonal basis of $H_k(\bb R)$ for any $k \in \bb N_0$. We have chosen the normalization constant of $\her_\ell$ in such a way that this sequence is actually an orthonormal basis of $L^2(\bb R)$.
The sequence $\{\her_{\ell}\}_{\ell \in \bb N_0}$ can be used in order to obtain an explicit description of the Sobolev spaces $H_{-k}(\bb R)$ of negative index. First we notice that
\[
\|u\|_k^2 = \sum_{\ell \in \bb N_0} (\ell +\tfrac{1}{2})^k \<u,\her_\ell\>^2.
\]
Therefore, the Sobolev space $H_{k}(\bb R)$ can be identified with the set of sequences $\{\hat{u}(\ell)\}_{\ell \in \bb N_0}$ such that
\[
\sum_{\ell \in \bb N_0} (\ell+\tfrac{1}{2})^k\hat{u}(\ell)^2 <+\infty.
\]
In fact, the mapping $\hat{u} \mapsto \sum_{\ell\in\ell_0} \hat{u}(\ell) \her_\ell$ is an isomorphism.
Then, the space $H_{-k}(\bb R)$ can be identified with the set of sequences $\{\hat{u}(\ell)\}_{\ell \in \bb N_0}$ such that
\[
\sum_{\ell \in \bb N_0} (\ell +\tfrac{1}{2})^{-k} \hat{u}(\ell)^2 <+\infty.
\]

We say that an $\mc S'(\bb R)$-valued random variable $\mc W$ is a {\em white noise} of mean zero and variance $\chi$ if for any function $u \in \mc S(\bb R)$, the real-valued random variable $\mc W(u)$ has a Gaussian distribution of mean zero and variance $\chi \<u,u\>$. Notice that by the Cram\'er-Wold device, for any $u_1,...,u_\ell \in \mc S(\bb R)$ the vector $(\mc W(u_1),...,\mc W(u_\ell))$ is a Gaussian vector of mean zero and covariances given by
\footnote{We use the generic notation $P$ and $E$ for probabilities and expectations of random variables, whenever their meaning is clarified by the context.
}
\[
E[\mc W(u_i)\mc W(u_j)] = \chi \<u_i,u_j\>.
\]
A possible way to construct a white noise $\mc W$ of variance $\chi$ is the following. Let $\{\zeta_i\}_{i \in \bb N_0}$ be a sequence of i.i.d. random variables of common distribution $\mc N(0,1)$. Define {\em formally} $\mc W = \sqrt{\chi}\sum_{i\in\bb N_0} \zeta_i \her_i$, meaning that
\[
\mc W(u) = \sqrt{\chi}\sum_{i \in \bb N_0} \zeta_i \<\her_i, u\>
\]
for any $u \in \mc S(\bb R)$. Notice that the sum above is convergent in $L^2(P)$ as soon as $u \in L^2(\bb R)$. Therefore, $\mc W(u)$ is $P-a.s.$ well-defined, but the set of full measure where $\mc W(u)$ is well-defined depends, in principle, on $u$. Notice that
\[
E\Big[\Big\|\!\!\sum_{\ell=m+1}^M \!\!\!\zeta_\ell \her_\ell\Big\|^2_{-2}\Big] = \chi \!\!\!\sum_{\ell =m+1}^M \!\!\!(\ell+\tfrac{1}{2})^{-2}
\]
and therefore the series $\sum_{\ell\in\bb N_0} \zeta_\ell \her_\ell$ is $a.s.$ summable in $H_{-2}(\bb R)$ and in particular in $\mc S'(\bb R)$.

\subsection{The infinite-dimensional Ornstein-Uhlenbeck process}
\label{s1.4}
In order to introduce some notations and concepts that will be necessary in order to define what we understand as a solution of the stochastic Burgers equation, we describe in this section the infinite-dimensional Ornstein-Uhlenbeck process. From now on, we consider the space of distributions $\mc S'(\bb R)$ equipped with its strong topology. For a given complete topological vector space $X$ and for $T>0$, we denote by $\mc C([0,T]; X)$ the space of continuous functions $u: [0,T] \to X$. Notice that $\mc C([0,T]; X)$ is a Banach space with respect to the uniform topology. When $X$ is the state space of a certain stochastic process, we call the elements of $\mc C([0,T]; X)$ {\em trajectories}.

\begin{proposition}
\label{p1.1.5}
A trajectory $\{y_t; t \in [0,T]\}$ belongs to $\mc C([0,T]; \mc S'(\bb R))$ if and only if $\{y_t(u); t \in [0,T]\}$ belongs to $\mc C([0,T]; \bb R)$ for any $u \in \mc S(\bb R)$.
\end{proposition}
\begin{proof}
See \cite{Tre}.
\end{proof}

We say that an $\mc S'(\bb R)$-valued stochastic process $\{\mc B_t; t \in [0,T]\}$ is an $\mc S'(\bb R)$-valued {\em Brownian motion} if for any $u \in \mc S(\bb R)$, the real-valued process $\{\mc B_t(u); t \in [0,T]\}$ is a Brownian motion of infinitesimal variance $\<u,u\>$.

\begin{remark}
One of the advantages of working with $\mc S'(\bb R)$-valued stochastic processes is the following observation. If an $\mc S'(\bb R)$-valued stochastic process  $\{\mc Y_t; t \in [0,T]\}$ is such that $\{\mc Y_t(u); t \in [0,T]\}$ has $a.s.$ continuous trajectories for any $u \in \mc S(\bb R)$, then it has a version with $a.s.$ continuous trajectories in $\mc S'(\bb R)$. In particular, we can always assume that an $\mc S'(\bb R)$-valued Brownian motion has continuous trajectories. This fact is a consequence of Mitoma's criterion, described in Proposition \ref{Mitoma} below.
\end{remark}

Let $D, \chi$ be some given positive constants. Let $\{ \mc B_t; t \in [0,T]\}$ be an $\mc S'(\bb R)$-valued Brownian motion. An explicit construction of $\{\mc B_t; t \in [0,T]\}$ is the following. Let $\{B_t^k; t \in [0,T]\}_{ k \in \bb N_0}$ be a sequence of independent, real-valued Brownian motions. Then we take
\[
\mc B_t(u) = \sum_{k \in \bb N_0} B_t^k \< u, \her_k\>
\]
for any $u \in \mc S(\bb R)$.

We say that an $\mc S'(\bb R)$-valued stochastic process $\{\mc Y_t; t \in [0,T]\}$ is a solution of the {\em infinite-dimensional Ornstein-Uhlenbeck equation}
\begin{equation}
\label{OU}
d \mc Y_t = D \Delta \mc Y_t dt + \sqrt{2\chi D}  \nabla d\mc B_t
\end{equation}
if for any trajectory $\{u_t; t \in [0,T]\}$ in $\mc S(\bb R)$, the process
\[
\mc Y_t(u_t)-\mc Y_0(u_0) - D\int_0^t \mc Y_s((\partial_s+\Delta)u_s)ds
\]
is a continuous martingale of quadratic variation
\[
2\chi D\int_0^t \<\nabla u_s,\nabla u_s\> ds.
\]
The quantity $\<\nabla u,\nabla u\>$ will appear repeatedly in what follows, so we introduce the notation $\mc E(u) = \<\nabla u, \nabla u\>$ for any $u \in \mc S(\bb R)$. We will call $\mc E(u)$ the {\em energy} of the function $u$.
Like in the case of the finite-dimensional Ornstein-Uhlenbeck process, for a given initial random distribution $\mc Y_0$, the solution of \eqref{OU} can be written explicitly in terms of the process $\mc B_t$. Let us denote by $\{P_t; t \in [0,T]\}$ the semigroup generated by the operator $D\Delta$ (that is, $P_t = e^{tD\Delta}$) and by $\{\nabla \mc B_t; t \in [0,T]\}$ the process given by $\nabla \mc B_t(u) = \mc B_t(\nabla u)$ for any $u \in \mc S(\bb R)$ and any $t \in [0,T]$. Then, the process $\{\mc Y_t; t \in [0,T]\}$ defined as
\[
\mc Y_t(u) = \mc Y_0(P_t u) + \sqrt{2\chi D}\int_0^t  \nabla d \mc B_s(P_{t-s}u)
\]
is a solution of \eqref{OU}. After some computations, we can see that the integral term is a Gaussian process with variance $\chi(\<u,u\>-\<P_tu,P_tu\>)$. Therefore, if the initial distribution $\mc Y_0$ is a white noise of variance $\chi$ and independent of $\{\mc B_t; t \in [0,T]\}$, we see that $\mc Y_t$ is also a white noise of variance $\chi$ for any later time $t \in (0,T]$. We say that the white noise of variance $\chi$ is a {\em stationary distribution} of the infinite-dimensional Ornstein-Uhlenbeck defined by \eqref{OU}.
\subsection{The stochastic Burgers equation}
\label{s1.5}

In this section we explain what we understand by a solution of the stochastic Burgers equation.
Let $D, \chi>0$, $\lambda \in \bb R$ be fixed constants. The stochastic Burgers equation is the equation
\begin{equation}
\label{SBE}
d \mc Y_t = D \Delta \mc Y_t dt + \lambda \nabla \mc Y_t^2 dt + \sqrt{2\chi D}  \nabla d \mc B_t,
\end{equation}
where $\{\mc B_t; t \in [0,T]\}$ is an $\mc S'(\bb R)$-valued Brownian motion. The nonlinear term $\nabla \mc Y_t^2$ is the source of all the trouble, not because of the nabla operator, but because of the square. In fact, from a heuristic point of view, the Burgers nonlinearity leaves white noise invariant, and therefore it is expected that the white noise of variance $\chi$ is also a stationary distribution of this equation, see \cite{Qua2}. If we assume this, it is reasonable to expect that solutions of \eqref{SBE} looks locally like white noise, in which case there is no convincing way to define the square $\mc Y_t^2$. The main goal of this section is to give a rigorous meaning to this square.

Justified by the discussion above and the one about stationary distributions of \eqref{OU}, we introduce the following definition.
\begin{definition}(Condition {\bf  (S)})\\
We say that a stochastic process $\{\mc Y_t; t \in [0,T]\}$ satisfies condition ($\mathbf{S}$) (from {\em stationarity}) if for any fixed time $t \in [0,T]$, the $\mc S'(\bb R)$-valued random variable $\mc Y_t$ is a white noise of variance $\chi$.
\end{definition}
For $\epsilon >0$ and $x \in \bb R$, let $i_\epsilon(x):  \bb R \to \bb R$ be defined as $i_\epsilon(x;y)= \epsilon^{-1} \textbf{1}_{x<y \leq x+\epsilon}$ for any $y \in \bb R$. Notice that the random variables $\mc Y_t(i_\epsilon(x))$ are well-defined if the process $\{\mc Y_t; t \in [0,T]\}$ satisfies condition {\bf  (S)}.

\begin{definition}(Condition {\bf (NL)})\\
We say that an $\mc S'(\bb R)$-valued stochastic process $\{\mc Y_t; t \in [0,T]\}$ satisfies condition {\bf(NL)} if:
\begin{itemize}
\item[{\bf (NL)}] there exists an $\mc S'(\bb R)$-valued stochastic process $\{\mc A_t; t \in [0,T]\}$ with continuous paths such that
\[
\mc A_t(u) = \lim_{\epsilon \to 0} \int_0^t \!\! \int\limits_{\bb R} \mc Y_s(i_\epsilon(x))^2 \nabla u(x) dx ds
\]
in $L^2$, for any $t \in [0,T]$ and any $u \in \mc S(\bb R)$.
\end{itemize}
\end{definition}

In a sense, this is the most na\"ive way in which the nonlinear term of the stochastic Burgers equation \eqref{SBE} could be defined. Once we have a way to define the nonlinearity in \eqref{SBE}, we can speak about solutions.

\begin{definition}(Stationary weak solution of \eqref{SBE})\\
We say that a  stochastic process $\{\mc Y_t; t \in [0,T]\}$ satisfying condition {\bf (NL)} is a stationary weak solution of \eqref{SBE} if for any $u \in \mc S(\bb R)$, the process
\[
\mc Y_t(u)- \mc Y_0(u) - D\int_0^t \mc Y_s(\Delta u) ds -\lambda\mc A_t(u)
\]
is a continuous martingale of quadratic variation $2 \chi D t \mc E(u)$.
\end{definition}
This notion of solution is too weak to be useful. The point is that it is not clear at all how to show any meaningful version of the It\^o's formula starting  just from this notion of solution. The energy condition allows to show some continuity properties of the quadratic term, which in particular, imply that the quadratic variation of it is zero. This allows to write down some version of the It\^o's formula for energy solutions which allow some manipulations. This point is further developed in \cite{GJ3}. In particular, the energy condition is needed in order to define properly the current fluctuation field. Because of this, we will introduce what we call the {\em energy estimate} in order to have well-behaved solutions of \eqref{SBE}.

Let $\{\mc Y_t; t \in [0,T]\}$ be a stationary stochastic process. For $s<t \in [0,T]$, $\epsilon >0$ and $u \in \mc S(\bb R)$, let us define
\[
\mc A_{s,t}^\epsilon(u) = \int_s^t \int\limits_{\bb R} \mc
Y_{r}(i_\epsilon(x))^2 \nabla u(x) dx dr.
\]
\begin{definition}(Energy estimate)\\
Let $\{\mc Y_t; t \in [0,T]\}$ be a process satisfying condition \textbf{(S)}.
We say that $\{\mc Y_t; t \in [0,T]\}$ satisfies the {\em energy estimate} if there exists a positive constant $\kappa$ such that
\begin{itemize}
\item[{\bf (EC1)}] For any $u \in \mc S(\bb R)$ and any $s<t \in [0,T]$,
\[
\bb E \Big[\Big(\int_s^t \mc Y_{r}(\Delta u) dr\Big)^2\Big] \leq \kappa (t-s)
\mc E(u),
\]
\item[{\bf (EC2)}] For any $u \in \mc S(\bb R)$, any $\delta<\epsilon \in (0,1)$ and any $s<t \in [0,T]$,
\[
\bb E\big[\big(\mc A_{s,t}^\epsilon(u)- \mc A_{s,t}^\delta(u)\big)^2\big] \leq \kappa(t-s) \epsilon \mc E(u).
\]
\end{itemize}
\end{definition}

\begin{remark}
We assumed that $\{\mc Y_t; t \in [0,T]\}$ is stationary only to make sense of $\mc Y_t(i_\epsilon(x))$ (recall that $i_\epsilon(x)$ is {\em not} a test function!). We can either relax the stationarity assumption to a weaker one which still allows to make sense of $\mc Y_t(i_\epsilon(x))$ or to consider a more regular approximation of the identity instead of $i_\epsilon$. But this will not be necessary for the goals of this article.
\end{remark}

We notice that, in the context of hydrodynamic limits, condition (EC1) is the equivalent of the energy condition for second-order partial differential equations. Moreover, also notice that {\bf (EC2)} implies that for any fixed $s,t$ and $u$, the random variables $\{\mc A_{s,t}^\epsilon(u)\}_{\epsilon>0}$ have a limit in $L^2$, as $\epsilon$ tends to $0$, which we can call $\mc A_t(u)-\mc A_s(u)$. But this does not imply {\em a priori} the existence of a {\em process} $\{\mc A_t; t \in [0,T]\}$ such that {\bf (NL)} is satisfied. This is the content of the following result.

\begin{theorem}
\label{t1.1.5}
Conditions {\bf (S)} and {\bf (EC2)} imply condition {\bf (NL)}.
\end{theorem}

Now we can finally state our main notion of solution of \eqref{SBE}:
\begin{definition}(Stationary energy solution of \eqref{SBE})\\
We say that
a stochastic process $\{\mc Y_t; t \in [0,T]\}$ is a stationary energy solution
of the stochastic Burgers equation \eqref{SBE} if
\begin{itemize}
\item[i)] The process $\{\mc Y_t; t \in [0,T]\}$
 is stationary and it satisfies the energy estimate.
\item[ii)] For any test function $u \in \mc S(\bb R)$ and any $t \in [0,T]$, the process
\[
\mc Y_t(u)- \mc Y_0(u) - D\int_0^t \mc Y_s(\Delta u) ds - \lambda\mc A_t(u)
\]
is a continuous martingale of quadratic variation $2 \chi D t \mc E(u)$, where
the process $\{\mc A_t; t \in [0,T]\}$ is obtained using Theorem \ref{t1.1.5}.
\end{itemize}

\end{definition}

The energy condition is strong enough to obtain some path properties of the stationary energy solutions of \eqref{SBE}:

\begin{theorem}
\label{t2.1.5} Let $\{\mc Y_t; t \in [0,T]\}$ be a stationary energy solution of \eqref{SBE}. For any function $u \in \mc S(\bb R)$, the process $\{\mc Y_t(u); t \in [0,T]\}$ is $a.s.$ H\"older-continuous of index $\alpha$, for any $\alpha <\frac{1}{4}$.
\end{theorem}

More properties of these solutions will be stated in the following section, in the context of the {\em KPZ equation}.

\subsection{The KPZ equation}
\label{s1.6}
In this section we introduce the celebrated {\em KPZ equation}, first introduced in \cite{KPZ}, which is intimately connected with the stochastic Burgers equation introduced in the previous section.
Let $D, \chi>0$, $\lambda \in \bb R$, be fixed constants. The KPZ equation is the equation
\begin{equation}
\label{KPZ}
d h_t = D \Delta h_t dt + \lambda (\nabla h_t)^2 dt + \sqrt{2 \chi D} d \mc B_t,
\end{equation}
where $\{\mc B_t; t \in [0,T]\}$ is an $\mc S'(\bb R)$-valued Brownian motion. Notice that at an heuristic level, if a process $\{h_t; t \in [0,T]\}$ is a solution of the KPZ equation \eqref{KPZ}, then the process $\{\mc Y_t, t \in [0,T]\}$ defined as $\mc Y_t(u) = \<h_t,\nabla u\>$ is a solution of the stochastic Burgers equation \eqref{SBE}.

In the original work of Kardar, Parisi and Zhang, the authors proposed a notion of solution of the KPZ equation \eqref{KPZ} which is nowadays known as the {\em Cole-Hopf} solution of \eqref{KPZ}. Take a solution $\{h(t,x); t \in [0,T], x \in \bb R\}$ of \eqref{KPZ} and define $z(t,x) = e^{\gamma h(t,x)}$, where $\gamma = \frac{\lambda}{D}$ (notice that we are implicitly assuming that $h_t$ is a function). Formal manipulations show that the process $\{z(t,x); t \in [0,T], x \in \bb R\}$ is a solution of the {\em stochastic heat equation}
\begin{equation}
\label{SHE}
d z_t = D \Delta z_t dt + \lambda \sqrt{\tfrac{2 \chi}{D}} z_t d\mc B_t.
\end{equation}
This equation, being linear, is easy to analyse. In particular, it can be proved that under mild assumptions on the initial data $\{z(0,x); x \in \bb R\}$, the corresponding Cauchy problem has a unique weak solution, which turns out to be continuous in time and space. One important initial distribution for which this existence-and-uniqueness result holds is $z(0,x) = e^{\gamma \chi B(x)}$, where $\{B(x); x \in \bb R\}$ is a two-sided, standard Brownian motion, for which in order to normalize it we take $B(0) \equiv 0$. Following \cite{KPZ}, we say that a stochastic process $\{h(t,x); t \in [0,T], x \in \bb R\}$ is a {\em Cole-Hopf solution} of \eqref{KPZ} if the process $\{z(t,x); t \in [0,T], x \in \bb R\}$ defined as $z(t,x) = e^{\gamma h(t,x)}$ is a weak solution of \eqref{SHE}. The physical relevance of Cole-Hopf solutions of \eqref{KPZ} was shown in \cite{BG}, where the authors showed that Cole-Hopf solutions of \eqref{KPZ} arise as the scaling limit of the density fluctuations of the weakly asymmetric, simple exclusion process (see Theorem \ref{p1.1.7} below), which corresponds to the model introduced in Section \ref{s1.2} with the simple choice $r \equiv 1$.

An important open question posed in \cite{BG} is to show that the Cole-Hopf solution solves the KPZ equation \eqref{KPZ} in any meaningful, direct sense
\footnote{
{\bf Note added in proof.} After the publication of a first draft of this work, and as explained in the introduction, M.~Hairer \cite{Hai} has provided a more satisfactory answer to this question.
}. In order to answer this question, we need to introduce what we call an {\em energy solution} of the KPZ equation \eqref{KPZ}.
This definition is just a translation of the definition of an energy solution of the stochastic Burgers equation \eqref{SBE}.
\begin{definition}(Condition ({\bf S'}))\\
We say that the real-valued stochastic process $\{h(t,x); t \in [0,T], x \in \bb R\}$ satisfies condition \textbf{(S')} if for any $t \in [0,T]$, the real-valued, spatial process $\{h(t,x); x \in \bb R\}$ is a two-sided Brownian motion of variance $\chi$.
\end{definition}
Notice that when we say that a spatial process is a two-sided Brownian motion, there is a one-parameter indetermination, caused by the arbitrary choice of its value at $x =0$. In our case, this arbitrary constant is just $h(t,0)$, which is a complicated process of unknown distribution. Notice that for a process $\{h(t,x); t \in [0,T],x \in \bb Z\}$ satisfying condition ({\bf S'}), the product $\<h_t^2,u\>$ is almost-surely well-defined for any function $u \in \mc S(\bb R)$.
Assume that the inner product $\<h_t^2,u\>$ is $a.s.$-well-defined for any $t \in [0,T]$ and any $u \in \mc S(\bb R)$ (which is the case for processes satisfying condition ({\bf S'})). For $s<t \in [0,T]$, $\epsilon>0$ and $u \in \mc S(\bb R)$, define
\[
\tilde{\mc A}_{s,t}^\epsilon(u) = \int_s^t \int\limits_{\bb R}
\Big\{\Big(\frac{h(r,x+\epsilon)-h(r,x)}{\epsilon}\Big)^2-\frac{\chi}{\epsilon}
\Big\} u(x) dx dr.
\]

Notice that in the case of a process satisfying condition \textbf{(S')}, the factor $\frac{\chi}{\epsilon}$ is just the expectation of the square in the definition above. This type of rescaling on which we take out an exploding constant is known in the literature as {\em Wick renormalization}.
\begin{definition}(Energy condition)\\
Let $\{h(t,x); t \in [0,T], x \in \bb R\}$ be a process satisfying condition \textbf{(S')}. We say that $\{h(t,x); t \in [0,T], x \in \bb R\}$ satisfies the {\em energy condition} if there exists a positive constant $\kappa$ such that
\begin{itemize}
\item[{\bf (EC1')}] For any $u \in \mc S(\bb R)$ and any $s<t \in [0,T]$,
\[
\bb E\Big[\Big(\int_s^t\int\limits_{\bb R} h(r,x) \Delta u(x) dx dr
\Big)^2\Big] \leq \kappa  \<u,u\>(t-s).
\]
\item[{\bf (EC2')}] For any $s<t \in [0,T]$, $\epsilon>\delta$ and $u \in \mc S(\bb R)$ such that $\int u(x) dx =0$,
\[
\bb E\big[\big(\tilde{\mc A}_{s,t}^\epsilon(u)-\tilde{\mc A}_{s,t}^\delta(u)\big)^2\big] \leq \kappa \epsilon \<u,u\> (t-s).
\]
\end{itemize}
\end{definition}

\begin{definition}(Almost stationary energy solution of \eqref{KPZ})\\
We say that a process $\{h(t,x); t \in [0,T], x \in \bb R\}$ is an almost stationary {\em energy solution} of the KPZ equation \eqref{KPZ} if:
\begin{itemize}
\item[i)] The process $\{h(t,x); t \in [0,T], x \in \bb R\}$ is almost stationary and it satisfies the energy condition.
\item[ii)] For any test function $u \in \mc S(\bb R)$ and any $t \in [0,T]$, the process
\[
\<h_t,u\>-\<h_0,u\>-D\int_0^t\<h_s,\Delta u\> ds - \lambda \tilde{\mc A}_{0,t}(u)
\]
is a continuous martingale of quadratic variation $2\chi D\<u,u\> t$, where
\[
\tilde{\mc A}_{0,t}(u) = \lim_{\epsilon \to 0} \tilde{\mc A}_{0,t}^\epsilon(u),
\]
and the limit is in $L^2$.
\end{itemize}
\end{definition}
The following theorem answers the open question posed by \cite{BG} in the almost stationary case:

\begin{theorem}
\label{t1.1.6}
The almost stationary Cole-Hopf solution of \eqref{KPZ} is also an almost stationary energy solution of the KPZ equation \eqref{KPZ}.
\end{theorem}

The relation between almost stationary solutions of the KPZ equation \eqref{KPZ} and stationary solutions of the stochastic Burgers equation \eqref{SBE} is given by the following observation.
Let $\{h(t,x); t \in [0,T], x \in \bb R\}$ be an almost stationary energy solution of the KPZ equation \eqref{KPZ}. Define the $\mc S'(\bb R)$-valued process $\{\mc Y_t; t \in [0,T]\}$ by taking $\mc Y_t(u) = \int_{\bb R} h(t,x) \nabla u(x) dx$ for any $t \in [0,T]$ and any $u \in \mc S(\bb R)$. Then, $\{\mc Y_t; t \in [0,T]\}$ is a stationary energy solution of the stochastic Burgers equation \eqref{SBE}.

Now we explain how to obtain an almost stationary energy solution of \eqref{KPZ} from a given stationary energy solution of \eqref{SBE}. Define $F,\theta: \bb R \to \bb R$ as $F(x) = (1+e^{-x})^{-1}$ and
\[
\theta(x) = \frac{\int_x^\infty e^{-\frac{1}{y(1-y)}} \textbf{1}_{0<y<1} dy}{\int_0^1 e^{-\frac{1}{y(1-y)}} dy}
\]
for any $x \in \bb R$. For $M \in \bb N$, define $\theta_M:\bb R \to \bb R$ as $\theta_M(x) = \theta(\frac{X}{M})$ for any $x \in \bb R$. Define $F_M: \bb R \to \bb R$ as $F_M(x) = F(x)\theta_M(x)$ for any $x \in \bb R$. Notice that
\begin{equation}
\label{ec2.5}
\lim_{M \to \infty} \mc E(F-F_M) = 0.
\end{equation}
In fact, it is enough to observe that $\nabla(F-F_M) = (1-\theta_M)\nabla F + \frac{1}{M} F \theta'(\frac{x}{M})$. For each function $u \in \mc S(\bb R)$, we denote by $U:\bb R \to \bb R$ the primitive of $u$ given by
\begin{equation*}
U(x) = \int_{-\infty}^x u(y) dy.
\end{equation*}
Notice that $\int_{\bb R} u(x) dx = U(\infty):=\bar u$.
The following theorem explains how to recover the integration constant:

\begin{theorem}
\label{t2.1.6}
Let $\{\mc Y_t; t \in [0,T]\}$ be a stationary energy solution of the stochastic Burgers equation \eqref{SBE}. Then, as $M$ tends to $\infty$, the sequence of real-valued processes $\{\mc Y_t(F_M)-\mc Y_0(F_M); t \in [0,T]\}_{M \in \bb N}$ converges in distribution with respect to the uniform topology of $\mc C([0,T]; \bb R)$, to a process which we call $\{\mc Y_t^\ast(F)-\mc Y_0^\ast(F); t \in [0,T]\}$. Moreover, for any $u \in \mc S(\bb R)$, the process $\{h_t; t \in [0,T]\}$ defined as
\begin{equation}
\label{def_h}
\begin{split}
h_0(x) &= \mc Y_0(\textbf{\em 1}_{(0,x]})\\
\<h_t, u\> = \mc Y_t(U-\bar u F)-\mc Y_0(U-&\bar u F) + \bar u\big(\mc Y_t^\ast(F)-\mc Y_0^\ast(F)\big) +\<h_0,u\>,
\end{split}
\end{equation}
is an almost stationary energy solution of \eqref{KPZ}.
\end{theorem}

\begin{remark}
If $u \in \mc S(\bb R)$ is a mean zero function, then $U \in \mc S(\bb R)$ and it is natural to define $\<h_t,u\> = \<h_0,u\> + \mc Y_t(U)-\mc Y_0(U)$. Therefore, if we are able to define $\<h_t,u\>$ for just one function with mean different from zero, we will be able to use linearity to define $\<h_t,u\>$ for any function $u \in \mc S(\bb R)$. This theorem allows us to do so for the function $\nabla F$.
\end{remark}
\begin{remark}
The height function $h(0,x)$ can be identified as the limit of the current of particles through the origin of the microscopic models. Formally, the current of particles through the origin can be identified with $\mc Y_t(\Theta_0)-\mc Y_0(\Theta_0)$, where $\Theta_0 = \textbf{\em 1}_{[0,\infty)}$ is the Heaviside function. Since $\Theta_0 \notin \mc S(\bb R)$, the process $\mc Y_t(\Theta_0)-\mc Y_0(\Theta_0)$ needs to be interpreted in some way. This was achieved by Rost and Vares in \cite{RV} by approximating $\Theta_0$ by the functions $\Theta_0(x)(1-\frac{x}{M})^+$. Theorem \ref{t2.1.6} is a smooth version of this argument, and it was in fact inspired by \cite{RV}.
\end{remark}

\subsection{The density fluctuation field}
\label{s1.7}
Let us assume for a moment that $a =0$. Later in this section we will deal with the case $a \neq 0$.
Recall that we are assuming that the process $\{\eta_t^n; t \geq 0\}$ has initial distribution $\nu_\rho$. Therefore, for any time $t \geq 0$, the configuration of particles $\eta_t^n$ also has distribution $\nu_\rho$. The scaling parameter $n$ represents the inverse of a spatial mesh, which allows us to see $\bb Z$ as a discrete approximation of the real line. The time acceleration $n^2$ introduced in the definition of $\{\eta_t^n; t \geq 0\}$ turns out to be the right one in order to see a non-trivial evolution of different functionals of $\{\eta_t^n; t \geq 0\}$ as $n$ tends to infinity. With this embedding in mind, we define the {\em density fluctuation field} $\{\mc Y_t^n; t \in [0,T]\}$ as the $\mc S'(\bb R)$-valued process given by
\[
\mc Y_t^n(u) = \frac{1}{\sqrt n} \sum_{x \in \bb Z} \big(\eta_t^n(x)-\rho\big) u(\tfrac{x}{n})
\]
for any test function $u \in \mc S(\bb R)$ and any $t \geq 0$. The main objective of this paper is the derivation of a scaling limit for this density fluctuation field. Notice that the decay properties of $u$ at infinity and the boundedness of $\eta_t^n(x)$ ensure that $\mc Y_t^n$ is well-defined as a distribution. Notice as well that the process $\{\mc Y_t^n; t \in [0,T]\}$ has trajectories in the space $\mc D([0,T];\mc S'(\bb R))$ of c\`adl\`ag paths in $\mc S'(\bb R)$.

Notice that the random variables $\{\eta(x); x \in \bb Z\}$ are i.i.d.~with respect to the measure $\nu_\rho$. In particular, it is not hard to see that for any fixed time $t \geq 0$, the random distribution $\mc Y_t^n$ converges in distribution, as $n$ tends to $\infty$, to $\sqrt{\rho(1-\rho)} \mc W$, where $\mc W$ is a standard white noise in $\mc S'(\bb R)$. The quantity $\chi(\rho) = \rho(1-\rho)$ is called the {\em static compressibility} of the system. When the initial number of particles is finite, their positions converge to a system of independent Brownian motions with reflection, under the time-space rescaling introduced above. Therefore, it is reasonable to guess that the density fluctuation field $\{\mc Y_t^n; t \in [0,T]\}$ converges to some limiting process when $n$ tends to $\infty$. This was actually proved in \cite{D-MPSW} for the model considered here, following the approach introduced in \cite{BR} in the context of the zero-range process. In order to describe the limiting object, we need to introduce a second parameter. Recall the definition of the local function $\omega: \Omega \to \bb R$ given in Sect.~\ref{s1.2}. For $\sigma \in [0,1]$, let us define $\varphi_\omega(\sigma) = \int \omega d \nu_\sigma$. Then, we define $D(\rho) = \varphi_\omega'(\rho)$. A simple computation shows that $D(\rho) = \int r d\nu_\rho$, although in other gradient models, like the zero-range process, there exists a function playing the role of $\omega$, but there is no function playing the role of $r$. In order to prove the previous equality, consider a Bernoulli product measure with density $\nu(\tfrac{x}{n})$, where $\nu(x) = \tfrac{\rho x}{1+(\rho x)^2}$ and compute the corresponding expectations. The constant $D(\rho)$ just introduced is called the {\em diffusivity} of the system. The scaling limit of the process $\{\mc Y_t^n; t \in [0,T]\}$ is described in the following proposition:

\begin{proposition}[\cite{D-MPSW}]
\label{p:ou}
As $n$ tends to $\infty$, the sequence of processes $\{\mc Y_t^n; t \in [0,T]\}_{n \in \bb N}$ converges in distribution with respect to the $J_1$-Skorohod topology of $\mc D([0,T];\mc S'(\bb R))$ to the stationary solution of the infinite-dimensional Ornstein-Uhlenbeck equation
\begin{equation}
\label{ou}
d\mc Y_t = D(\rho) \Delta \mc Y_t dt + \sqrt{2\chi(\rho)D(\rho)} \nabla d\mc B_t,
\end{equation}
where $\{\mc B_t; t \in [0,T]\}$ is an $\mc S'(\bb R)$-valued Brownian motion.
\end{proposition}

The process $\{\mc Y_t; t \in [0,T]\}$ was defined in Section \ref{s1.4}. In this article we are interested on the case $a \neq 0$, although the techniques that we will introduce allow to prove some new results in the case case $a=0$ as well (see \cite{GJ}). Therefore, let us return to the general case $a \in \bb R$. For each $x \in \bb Z$, let $J_x^n(t)$ denote the cumulative current of particles between $x$ and $x+1$ up to time $t$, that is $J_x^n(t)$ is equal to the number of particles passing from $x$ to $x+1$ up to time $t$, minus the number of particles passing from $x+1$ to $x$ up to time $t$. The process $\{J_x^n(t); t \geq 0\}$ is a compound Poisson process of instantaneous rate
\[
j_x^n(t) = n^2 r_x(\eta_t^n)\Big(1+ \tfrac{a}{\sqrt n} \eta_t^n(x+1)\big(1-\eta_t^n(x)\big)\Big) \big(\eta_t^n(x+1)-\eta_t^n(x)\big).
\]
Notice that the expectation of the instantaneous current  $j_x^n(t)$ with respect to $\nu_\rho$ is equal to $a n^{3/2} \chi(\rho) D(\rho)$. Let us define the {\em flux}  $H(\rho) = a \chi(\rho) D(\rho)$. Aside from the scaling factor $n^{3/2}$, the flux governs the transport of density fluctuations. In fact, at density $\rho$, a fluctuation travels at velocity $n^{3/2} H'(\rho)$. Therefore, in order to see a non-trivial evolution of the density fluctuations, we need to {\em recenter} the density fluctuation field. This observation leads us to define $\{\mc Y_t^n; t \in [0,T]\}$ as
\[
\mc Y_t^n(u) = \frac{1}{\sqrt n} \sum_{x \in \bb Z} \big(\eta_t^n(x)-\rho\big) u\Big( \frac{x-n^{3/2}H'(\rho)t}{n}\Big)
\]
for any $u \in \mc S(\bb R)$. Due to the ellipticity condition i), $D(\rho)$ is bounded above and below. Moreover, the finite-range condition ii) implies that $D(\rho)$ is a polynomial. Since $H(0)=H(1)=0$, there exists a density $\rho \in (0,1)$ such that $H'(\rho)=0$, and for that particular choice of the density, there is no recentering in the definition of $\mc Y_t^n$. In order to simplify the notation and computations, we will assume from now on that $H'(\rho)=0$; the results of the article hold for any choice of $\rho \in (0,1)$ with basically notational modifications. We will also assume that $H''(\rho) \neq 0$; we will comment about this assumption after stating our main Theorem \ref{t0}.

The particular case $r \equiv 1$ corresponds to the weakly asymmetric simple exclusion process, which is among the most studied interacting particle systems. In that particular case, the scaling limit of $\{\mc Y_t^n; t \in [0,T]\}$ was obtained in \cite{BG}, even for some non-stationary initial distributions. Since in this article we are always in a stationary situation, we only state the result of \cite{BG} for the stationary case:

\begin{proposition}
\label{p1.1.7} Assume that $r \equiv 1$.
Then the sequence of  processes $\{\mc Y_t^n; t \in [0,T]\}_{n \in \bb N}$ converges in distribution with respect to the $J_1$-Skorohod topology of $\mc D([0,T]; \mc S'(\bb R))$ to the process $\{\mc Y_t; t \in [0,T]\}$ defined as
\[
\mc Y_t(u) = -\frac{1}{a}\int\limits_{\bb R}  u'(x) \log z(t,x) dx
\]
for any $u \in \mc S(\bb R)$, where $\{z(t,x); t \in [0,T], x \in \bb R\}$ is the solution of the stochastic heat equation
\[
d z_t = \Delta z_t dt + az_t\sqrt{2\chi(\rho)} d\mc B_t
\]
with initial distribution $z(0,x) = \exp\{a B(x)\}$, where $\{B(x); x \in \bb R\}$ is a two-sided, standard Brownian motion and $\{\mc B_t; t \in [0,T]\}$ is an $\mc S'(\bb R)$-valued Brownian motion, independent of $\{B(x); x \in \bb R\}$.
\end{proposition}

Recent advances in combinatoric properties of the asymmetric simple exclusion process \cite{TW} led to an intense research activity based upon \cite{BG} (see \cite{SS}, \cite{ACQ}, and see \cite{Cor} for a recent review). However, the biggest disadvantage of the approach introduced in \cite{BG} is the use of the so-called {\em microscopic Cole-Hopf} transformation, discovered by Gartner \cite{Gar}. This transformation relates the asymmetric simple exclusion process with another particle system, for which the fluctuation density field satisfies a discretized version of the stochastic heat equation. This microscopic Cole-Hopf transformation does not work for any other type of interacting particle system. Therefore, in order to treat general interacting particle systems, another approach is needed. Our main theorem is the following.

\begin{theorem}
\label{t0}
The sequence of processes $\{\mc Y_t^n; t \in [0,T]\}_{n \in \bb N}$ is tight with respect to the $J_1$-Skorohod topology of $\mc D([0,T]; \mc S'(\bb R))$. Moreover, any limit point of $\{\mc Y_t^n; t \in [0,T]\}_{n \in \bb N}$ is a stationary energy solution of the stochastic Burgers equation
\begin{equation}
d \mc Y_t = D(\rho) \Delta \mc Y_t dt +\tfrac{1}{2} H''(\rho) \nabla \mc Y_t^2 dt + \sqrt{ 2 \chi(\rho)D(\rho)}  \nabla d\mc B_t,
\end{equation}
where $\{\mc B_t; t \in [0,T]\}$ is an $\mc S'(\bb R)$-valued Brownian motion.
\end{theorem}

\begin{remark}
As shown in \cite{BG}, the scale at which the asymmetric part of the dynamics induces a non-trivial evolution in the limit, is the weakly asymmetric scale $n^{3/2}$. A careful analysis of the proof of this theorem shows that if we take $a=a_n \to 0$ as $n$ tends to $\infty$, the sequence $\{\mc Y_t^n; t \in [0,T]\}_{n \in \bb N}$ converges to a stationary solution of the Ornstein-Uhlenbeck equation \eqref{ou}. This fact can be derived by adapting the proof of Theorem 2.6 of \cite{Gon} to our case.  If $H''(\rho)=0$, Theorem \ref{t0} states that the Ornstein-Uhlenbeck process given by \eqref{ou} is still the limit of the fluctuation field $\{\mc Y_t^n; t \in [0,T]\}_{n \in \bb N}$. For simplicity we assume $H''(\rho) \neq 0$.
\end{remark}

\subsection{The current fluctuation field and the height function}
\label{s1.8}

Recall the definition of the current $\{J_t^n(x); t \in [0,T]\}$ given in Section \ref{s1.7}. The expectation of $J_t^n(x)$ with respect to $\nu_\rho$ is equal to $n^{3/2} H(\rho)t$. Let us define the {\em current fluctuation field} as the $\mc S'(\bb R)$-valued process $\{\mc J_t^n; t \in [0,T]\}$ given by
\[
\mc J_t^n(u) = \frac{1}{n^{3/2}} \sum_{x \in \bb Z} \big(J_t^n(x)- n^{3/2}H(\rho)t\big)u(\tfrac{x}{n}).
\]
Notice that the conservation the number of particles and the fact that particles only jump to nearest-neighbor sites is reflected into the {\em continuity relation}
\[
\eta_t^n(x)-\eta_0^n(x) = J_t^n(x-1)-J_t^n(x),
\]
true for any $x \in \bb Z$ and any $t \in [0,T]$.
The scaling $n^{3/2}$ is explained by the following observation: for any function $u \in \mc S(\bb R)$, the continuity relations shows that
\[
\mc Y_t^n(u) - \mc Y_0^n(u) = \frac{1}{n^{3/2}} \sum_{x \in \bb Z}  J_t^n(x)\nabla_x^n u,
\]
where
\[
\nabla_x^n u = n\big\{u(\tfrac{x+1}{n})-u(\tfrac{x}{n})\big\}.
\]

Closely related to the current fluctuation field $\{\mc J_t^n; t \in [0,T]\}$, the {\em height fluctuation field} $\{\mc H_t^n(x); t \in [0,T], x \in \bb R\}$ is defined as follows. The height function $\{h_t^n(x); t \in [0,T], x \in \bb Z\}$ is defined as
\begin{equation}
\label{height1}
h_t^n(x) = J_t^n(x)-\sum_{i=1}^x \eta_0^n(i),
\end{equation}
which by the continuity relation is equivalent to
\begin{equation}
\label{height2}
h_t^n(x) = J_t^n(0)-\sum_{i=1}^x \eta_t^n(i).
\end{equation}
The process $\{h_t^n(x); t \in [0,T], x \in \bb Z\}$ can be interpreted as the height of a {\em growth interface}, and it was in this context that the KPZ was originally introduced in \cite{KPZ}. In this article, this relation is marginally relevant, so we do not go further on it and we refer to the review \cite{Cor} for a more detailed exposition.
For $x \in \bb Z$ and $t \in [0,T]$ we define
\begin{equation}
\mc H_t^n(\tfrac{x}{n}) = \frac{1}{\sqrt n} \big(h_t^n(x) - n^{3/2}H(\rho)t+\rho x\big),
\end{equation}
that is $\mc H_t^n(\frac{x}{n})$ is the centered, rescaled height function.
For arbitrary $x \in \bb R$, we define $\mc H_t^n(x)$ by linear interpolation. From \eqref{height2} we see that for any fixed time $t \in [0,T]$, the spatial process $\{\mc H_t^n(x)-\mc H_t^n(0); x \in \bb R\}$ is a rescaled random walk. Therefore, Donsker's Theorem shows that for any fixed time $t \in [0,T]$, the sequence $\{\mc H_t^n(x)-\mc H_t^n(0); x \in \bb R\}_{n \in \bb N}$ converges in distribution to a two-sided Brownian motion of variance $\chi(\rho)$. Recall that we are assuming that $H'(\rho)=0$. The following theorem is the analogue of Theorem \ref{t0} for the height fluctuation field.

\begin{theorem}
\label{t0.1}
The sequence of processes $\{\mc H_t^n(x); t \in [0,T], x \in \bb R\}_{n \in \bb N}$ is tight with respect to the $J_1$-Skorohod topology of $\mc D([0,T]; \mc S'(\bb R))$. Moreover, any limit point of $\{\mc H_t^n(x); t \in [0,T], x \in \bb R\}_{n \in \bb N}$ is an almost stationary solution of the KPZ equation
\[
d h_t = D(\rho) \Delta h_t dt + \tfrac{1}{2}aH''(\rho)\big(\nabla h_t\big)^2 dt +\sqrt{2\chi(\rho)D(\rho)}d \mc B_t,
\]
where $\{\mc B_t; t \in [0,T]\}$ is an $\mc S'(\bb R)$-valued Brownian motion.
\end{theorem}

\section{General tools}
\label{s2}
Let us recall that in Section \ref{s1.2} we fixed a density $\rho \in (0,1)$. Since in this section we will need to deal with arbitrary densities in $[0,1]$, we will use the denomination $\sigma$ for a generic density of particles in $[0,1]$. We reserve the denomination $\rho$ for the density already fixed in Section \ref{s1.2}.
\subsection{Spectral inequalities}
\label{s2.1}
For each $\sigma \in [0,1]$, let $L^2(\nu_\sigma)$ be the Hilbert space associated to the measure $\nu_\sigma$, namely, the space of functions $f: \Omega \to \bb R$ such that $\int f^2 d\nu_\sigma<+\infty$. We will denote by $\<f,g\>_\sigma$ the inner product in $L^2(\nu_\sigma)$, that is, $\<f,g\>_\sigma = \int fg d\nu_\sigma$ for any $f,g \in L^2(\nu_\sigma)$. Let us define, for $f \in L^2(\nu_\rho)$, the $H_{-1,n}$-norm of $f$ through the variational formula
\[
\|f\|^2_{-1,n} = \sup_{g \text{ local}} \big\{ 2\<f,g\>_\rho-n^2 \<g,-L_ng\>_\rho\big\}.
\]

Let $L_n^*$ be the adjoint of $L_n$ with respect to $\nu_\rho$. A simple computation shows that $L_n^*$ corresponds to the generator of a weakly asymmetric, simple exclusion process with speed change on which the jumps {\em to the right} of $x$ happen with an additional rate $\frac{a}{\sqrt n} r_x$. Notice that $\<g,-L_ng\>_\rho=\<g,-L_n^*g\>_\rho=\<g,-L_n^sg\>_\rho$, where $L_n^s= \frac{1}{2}(L_n+L_n^*)$ is the symmetric part of the generator $L_n$. By the description made above of $L_n^*$, we see that $L_n^s=(1+\frac{a}{2\sqrt n})S$, and therefore
\begin{equation}
\label{ec1}
\|f\|_{-1,n}^2 = \sup_{g \text{ local}} \Big\{ 2\<f,g\>_\rho-n^2\Big(1+\tfrac{a}{2\sqrt n} \Big) \<g,-Sg\>_\rho\Big\}.
\end{equation}

Observe that $\<g,-Sg\>_\rho= \<g+c,-S(g+c)\>_\rho$ for any constant $c \in \bb R$. This is a consequence of the invariance of the measure $\nu_\rho$ with respect to $S$. Therefore, if $\int f d\nu_\rho \neq 0$, then $\|f\|_{-1,n}^2=\infty$.

The relevance of the $H_{-1,n}$-norm is given by the following proposition:

\begin{proposition}[Kipnis-Varadhan inequality \cite{KV}, \cite{CLO}]
\label{p1}
Let $f: [0,T] \to L^2(\nu_\rho)$. Then,
\begin{equation}
\label{KV}
\bb E_n\Big[\Big(\sup_{0 \leq t \leq T} \int_0^t f(s,\eta_s^n) ds \Big)^2\Big] \leq 14 \int_0^T \|f(t,\cdot)\|_{-1,n}^2 dt.
\end{equation}
\end{proposition}

Inequality \eqref{KV} was introduced in \cite{KV} in the context of {\em reversible} Markov chains. A proof of this inequality in the form presented here, for which the reversibility of $\nu_\rho$ is not required, can be found in \cite{CLO}.

Considering the test function $\tilde g = n^{-2}(1+\frac{a}{\sqrt n})^{-1} g$ in the variational formula \eqref{ec1} we see that $\|f\|_{-1,n}^2 = n^{-2}(1+\frac{a}{\sqrt n})^{-1} \|f\|_{-1}^2$, where the $H_{-1}$-norm $\|f\|_{-1}^2$ is defined as
\[
\|f\|_{-1}^2 = \sup_{g \text{ local}} \big\{2\<f,g\>_\rho-\<g,-Sg\>_\rho\big\}.
\]
Therefore, inequality \eqref{KV} can be recast as
\[
\bb E_n \Big[ \Big(\sup_{0 \leq t \leq T} \int_0^t f(s,\eta_s^n) ds \Big)^2 \Big] \leq \frac{c_0}{n^2} \int_0^T \|f(t,\cdot)\|_{-1}^2 dt,
\]
where $c_0$ is a constant which depends only on $a \in \bb R$.

\begin{remark}
Although very useful in many situations, in this article we will not take advantage of the supremum  inside the expectation in \eqref{KV}.
\end{remark}

It is clear that inequality \eqref{KV} is not very useful unless we have an effective way to estimate $\|f\|_{-1}^2$. Due to the variational formula for $\|f\|_{-1}^2$, the following proposition proves to be very useful:

\begin{proposition}[Spectral gap inequality \cite{Qua,DS-C}]
\label{p2}
There exists a universal constant $c_1$ such that for any $\ell \in \bb N$ and any $f:\Omega \to \bb R$ such that $\supp(f) \subseteq \{1,...,\ell\}$ and such that $\int f d\nu_\sigma =0$ for any $\sigma \in [0,1]$,
\begin{equation}
\label{SG}
\<f,f\>_\rho \leq c_1 \ell^2 \sum_{x=1}^{\ell-1} \int \big( \nabla_{x,x+1} f\big)^2 d\nu_\rho.
\end{equation}
\end{proposition}

Notice that for any local function $f: \Omega \to \bb R$ we have that
\[
\<f,-Sf\>_\rho = \frac{1}{2} \sum_{x \in \bb Z} \int r_x
\big(\nabla_{x,x+1} f\big)^2 d\nu_\rho,
\]
and therefore the ellipticity condition i) shows that
\[
\sum_{x=1}^{\ell-1} \int \big(\nabla_{x,x+1}f\big)^2 d \nu_\rho \leq \sum_{x=1}^{\ell-1}\int \frac{r_x}{\epsilon_0} \big(\nabla_{x,x+1} f\big)^2 d\nu_\rho \leq \frac{2}{\epsilon_0} \<f,-Sf\>_\rho.
\]
In particular, from \eqref{SG} we conclude that
\begin{equation}
\label{ec3}
\<f,f\>_\rho  \leq \frac{2c_1 \ell^2}{\epsilon_0} \<f,-Sf\>_\rho
\end{equation}
for any function $f:\Omega \to \bb R$ satisfying the hypothesis of Proposition~\ref{p2}.

The following proposition shows how to use \eqref{ec3} in order to estimate the $H_{-1}$-norm of a local function $f$.

\begin{proposition}
\label{p3}
Let $f:\Omega \to \bb R$ be a local function, such that $\int f d\nu_\sigma =0$ for any $\sigma \in [0,1]$. Let $x \in \bb Z$ and $\ell \in \bb N$ be such that $\text{\em supp}(f) \subseteq \{x+1,...,x+\ell\}$. Then,
\[
\|f\|_{-1}^2 \leq \frac{2c_1 \ell^2}{\epsilon_0} \<f,f\>_\rho.
\]
\end{proposition}

Although this proposition is part of the {\em folklore} of interacting particle systems, we did not find any reference with the exact form of it needed in this article. Since this proposition is the key point in what follows, we present a proof of it.
\begin{proof}
By translation invariance of the dynamics and of the measure $\nu_\rho$, we can
assume without loss of generality that $x=0$. Let $g:\Omega \to \bb R$ be a
local function. Let $\mc F_{\ell} = \sigma(\eta(1),...,\eta(\ell))$ and define
$g_\ell = E[g|\mc F_\ell]$. Here and in what follows, all conditional
expectations are taken with respect to $\nu_\rho$. Recall that the value of
$\rho$ was fixed in Section \ref{s1.2}. For $x \in \bb Z$ and $\ell \in \bb N$,
let us define
\[
\eta^\ell(x) = \frac{1}{\ell} \sum_{i=1}^\ell \eta(x+i),
\]
and define $\bar{g}_\ell = g_\ell - E[g_\ell|\eta^\ell(0)]$. Since $\int f d\nu_\sigma=0$ for any $\sigma \in [0,1]$ and $\text{supp}(f) \subseteq \{1,...,\ell\}$, we have $E[f|\eta^\ell(0)]=0$ and therefore $\<f,g\>_\rho=\<f,\bar{g}_\ell\>_\rho$. On the other hand, since the integral $\int (\nabla_{x,x+1} g)^2 d\nu_\rho$ is convex as a function of $g$, we have that
\[
\sum_{x=1}^{\ell-1} \int\big(\nabla_{x,x+1} \bar{g}_\ell\big)^2 d\nu_\rho \leq \frac{2}{\epsilon_0} \<g,-Sg\>_\rho.
\]

Notice that $\bar{g}_\ell$ satisfies the hypothesis of Proposition \ref{p2}, while $g$ may not. Using the weighted Cauchy-Schwarz inequality we see that
\begin{align*}
2\<f,g\>_\rho
  &= 2\<f,\bar{g}_\ell\>_\rho
  \leq \beta \<f,f\>_\rho+ \frac{1}{\beta}\<\bar{g}_\ell,\bar{g}_\ell\>_\rho\\
  &\leq \beta \<f,f\>_\rho + \frac{c_1 \ell^2}{\beta} \sum_{x=1}^{\ell-1} \int\big(\nabla_{x,x+1} \bar{g}_\ell\big)^2 d\nu_\rho\\
  &\leq \beta\<f,f\>_\rho + \frac{2c_1 \ell^2}{ \epsilon_0 \beta} \<g,-Sg\>_\rho.
\end{align*}
Choosing $\beta = 2c_1 \ell^2 \epsilon_0^{-1}$, we see that
\[
2\<f,g\>_\rho- \<g,-Sg\>_\rho \leq \frac{2c_1 \ell^2}{\epsilon_0} \<f,f\>_\rho
\]
for any local function $g:\Omega \to \bb R$, which proves the proposition.
\end{proof}

Following the proof of Proposition \ref{p3}, we obtain the following result, which roughly states that functions with disjoint supports are orthogonal with respect to the $H_{-1}$-norm.

\begin{proposition}
\label{p4}
Let $m \in \bb N$ be given. Take a sequence $k_0 < ...<k_m$ in $\bb Z$ and let $\{f_1,...,f_m\}$ be a sequence of local functions from $\Omega$ to $\bb R$ such that $\supp(f_i) \subseteq \{k_{i-1}+1,...,k_i\}$ for any $i \in \{1,...,m\}$. Define $\ell_i = k_i-k_{i-1}$. Assume that $\int f_i d\nu_\sigma=0$ for any $\sigma \in [0,1]$ and any $i \in \{1,...,m\}$. Then,
\[
\|f_1+...+f_m\|_{-1}^2 \leq \sum_{i=1}^m \frac{c_1 \ell_i^2}{\epsilon_0} \<f_i,f_i\>_\rho.
\]
\end{proposition}
\begin{proof}
Let us define $\mc F_i = \sigma(\eta(k_{i-1}+1),...,\eta(k_i))$. Let $g:\Omega \to \bb R$ be a local function, define $g^i = E[g|\mc F_i]$ and $\bar{g}^i =g^i - E[g^i|\eta^{\ell_i}(k_{i-1})]$, and let $f=f_1+\cdots+f_m$. We have that
\[
\<f,g\>_\rho = \sum_{i=1}^m \<f_i,\bar{g}^i\>_\rho
\]
and by Jensen's inequality we have that
\[
\begin{split}
\<g,-Sg\>_\rho
		&\geq \sum_{i=1}^m \sum_{x=k_{i-1}+1}^{k_i-1} \int r_x\big( \nabla_{x,x+1}g)^2 d\nu_\rho \\
		&\geq \sum_{i=1}^m \sum_{x=k_{i-1}+1}^{k_i-1} \int r_x\big( \nabla_{x,x+1}\bar{g}^i)^2 d\nu_\rho.
\end{split}
\]
Let us write $S^i$ for the generator $S$ restricted to the interval $\{k_{i-1}+1,...,k_i\}$: for any function $f:\Omega \to \bb R$ and any $\eta \in \Omega$,
\[
S^i f(\eta) = \!\!\!\sum_{x=k_{i-1}+1}^{k_i-1} \!\!\! r_x(\eta)\nabla_{x,x+1} f(\eta).
\]
Following the proof of Proposition \ref{p3} we obtain that
\[
\begin{split}
\|f\|_{-1}^2= \sup_{g \text{ local}} \big\{ 2\<f,g\>_\rho -\<g,-Sg\>_\rho\big\}
		&\leq \sup_{g \text{ local}} \sum_{i=1}^m \big\{ 2\<f_i,\bar{g}^i\>_\rho -\<\bar{g}^i,-S^i\bar{g}^i\>_\rho\}\\
		&\leq \sum_{i=1}^m \sup_{g \in \mc F_i}\big\{  2\<f_i,\bar g\>_\rho - \frac{\epsilon_0}{c_1 \ell_i^2} \<\bar g,\bar g\>_\rho\big\}\\
		& \leq \sum_{i=1}^m \frac{c_1 \ell^2_i}{\epsilon_0} \<f_i,f_i\>_\rho.
\end{split}
\]
\end{proof}
Combining Propositions \ref{p1} and \ref{p4}, we obtain the following estimate:

\begin{corollary}
\label{c1}
Let $\{f_1,...,f_m\}$ a sequence of functions from $[0,T]\times \Omega$ to $\bb R$, satisfying the hypothesis of Proposition \ref{p4} for a given sequence $\{k_0,...,k_m\}$. Then,
\[
\bb E_n \Big[ \Big( \int_0^t \sum_{i=1}^m f_i(s,\eta_s^n)
ds \Big)^2 \Big] \leq \sum_{i=1}^m\frac{c_0 c_1 \ell_i^2}{\epsilon_0 n^2}
\int_0^t \<f_i(s,\cdot),f_i(s,\cdot)\>_\rho ds,
\]
where the constants $c_0$, $c_1$ are defined above.
\end{corollary}

\begin{remark}
In what follows, we will only use Corollary \ref{c1} and we will not make any explicit use of Propositions \ref{p1} and \ref{p4}. In particular, if for any specific model we find out a way to prove Corollary \ref{c1} without making use of Propositions \ref{p1} and \ref{p4}, the results of this article should be true modulo the usual technicalities inherent to the aforementioned specific model.
\end{remark}

\subsection{Equivalence of ensembles}
Let $x \in \bb Z$ and $\ell \in \bb N$. Let us recall the definition of $\eta^\ell(x)$.
\[
\eta^\ell(x) = \frac{1}{\ell} \sum_{i=1}^\ell \eta(x+i).
\]

Let $f: \Omega \to \bb R$ be a  local function, and let $x_0 \in \bb Z$, $\ell_0 \in \bb N$ be such that $\supp(f) \subseteq \{x_0+1,...,x_0+\ell_0\}$. For $\ell \geq \ell_0$ we define $\psi_f(\ell): \Omega \to \bb R$ by
\[
\psi_f(\ell; \eta) = E\big[f|\eta^\ell(x_0)\big].
\]
Notice that the definition of $\psi_f(\ell)$ depends on the choice of $x_0$,
although it does not depend on the choice of $\ell_0$ as soon as $\ell_0 \leq
\ell$. However, for different choices of $x_0$, the corresponding functions
differ only by a spatial translation. In what follows, the
choice of $x_0$ and $\ell_0$ do not matter, as long as it is kept fixed.

For $\sigma \in [0,1]$, let us define $\varphi_f(\sigma)=\int f d \nu_\rho$. The following proposition, known in the literature as {\em equivalence of ensembles}, gives us an approximation of $\psi_f(\ell)$ in terms of $\varphi_f$ and $\eta^\ell(x_0)$:

\begin{proposition}[Equivalence of ensembles]
\label{p5}
Let $f: \Omega\to \bb R$ be a  local function. Let $x_0 \in \bb Z$ and $\ell_0 \in \bb N$ be such that $\supp(f) \subseteq \{x_0+1,...,x_0+\ell_0\}$.
There exists a constant $c=c(f)$ such that for any $\ell \geq \ell_0$,
\[
\sup_{\eta \in \Omega}\Big| \psi_f(\ell;\eta) - \varphi_f(\eta^\ell(x_0))-{
\frac{1}{2\ell}}\chi(\eta^\ell(x_0))\varphi_f''(\eta^\ell(x_0))\Big|\leq
\frac{c}{\ell^2},
\]
where $\chi(\sigma) = \sigma(1-\sigma)$.
\end{proposition}

\begin{remark}
The function $\chi(\sigma)= \sigma(1-\sigma)$ is the static compressibility introduced in Section \ref{s1.7}.
\end{remark}

For a proof of this estimate, see Proposition 3.1 of \cite{GJ}.
Proposition \ref{p5} can be used to obtain the following estimate:

\begin{proposition}
\label{p6}
Under the hypothesis of Proposition~\ref{p5}, there exists a constant
$c=c(f,\rho)$ such that
\begin{multline*}
\int \Big\{\psi_f(\ell;\eta) -\varphi_f(\rho) - \varphi_f'(\rho)\big(\eta^\ell(x_0)-\rho\big)-\\
 - \frac{\varphi_f''(\rho)}{2}\Big(\big(\eta^\ell(x_0)-\rho\big)^2-\frac{\chi(\rho)}{\ell}\Big)\Big\}^2\nu_\rho(d\eta) \leq \frac{c}{\ell^3}
\end{multline*}
for any $\ell \geq  \ell_0$. In particular, we can choose $c$ in such a way that
\begin{itemize}
\item[i)] if $\varphi_f(\rho)=0$, then $\int \psi_f(\ell)^2 d\nu_\rho \leq c \ell^{-1}$,
\item[ii)] if $\varphi_f(\rho)=\varphi_f'(\rho)=0$, then $\int \psi_f(\ell)^2 d\nu_\rho \leq c \ell^{-2}$,
\item[iii)] if $\varphi_f(\rho)=\varphi'_f(\rho)=\varphi_f''(\rho)=0$, then $\int \psi_f(\ell)^2 d\nu_\rho \leq c \ell^{-3}$.
\end{itemize}
\end{proposition}

\section{Second-order Boltzmann-Gibbs principle}
\label{s3}
In this section we prove the main technical and theoretical innovation of this article, namely what we call the {\em second-order} Boltzmann-Gibbs principle. In order to make the exposition more transparent, we begin the discussion with the {\em usual} (or first-order) Boltzmann-Gibbs principle say.

\subsection{The Boltzmann-Gibbs principle}
\label{s3.1}

In the 80's, H.~Rost \cite{BR} introduced the celebrated {\em Boltzmann-Gibbs principle}. We quote \cite{D-MPSW} for a proof for the models considered in this article in the reversible case, and \cite{CLO} for a proof in the non-reversible case. Let $f: \Omega \to \bb R$ be a local function and recall the notation $\varphi_f(\sigma)= \int f d\nu_\sigma$. We have the following

\begin{proposition}[Boltzmann-Gibbs principle]
\label{p7}
For any continuous function $u: \bb R \to \bb R$ of compact support and any $t \geq 0$,
\begin{equation}
\label{ec7}
\lim_{n \to \infty} \bb E_n\Big[\Big(\int_0^t \frac{1}{\sqrt n} \sum_{x \in \bb Z} \big(\tau_x f(\eta_s^n)- \varphi_f(\rho) -\varphi'_f(\rho)\big(\eta_s^n(x)-\rho\big)\big)u(\tfrac{x}{n}) ds \Big)^2\Big] =0.
\end{equation}
\end{proposition}

Roughly speaking, this proposition says the following. Space-time fluctuations of the field associated to the function $f$ are asymptotically equivalent to a multiple of the density fluctuation field. Let us try to give a more intuitive justification for this formula. Particles are neither destroyed nor annihilated by the dynamics. Therefore, in order to equilibrate a local fluctuation on the number of particles, it is necessary to transport it to another region. The density of particles is the only locally conserved quantity of the system. Due to the ellipticity condition i), the process has good ergodic properties. Therefore, a fluctuation of a non-conserved quantity will be locally equilibrated. If we look at the process on the right time window, the only observed fluctuations will be the density fluctuations; other fluctuations, being too fast to be observed at that time window, are averaged out.

Notice that in the case $\varphi_f'(\rho)=0$, Proposition \ref{p7} does not give a lot of information: it simply states that the fluctuation field associated to $f$ asymptotically vanishes. In particular, it does not identify a possible scaling at which non-trivial fluctuations may be observed. In addition, Proposition \ref{p7} does not give any quantitative estimate about the speed of convergence in \eqref{ec7}. These two points will be addressed in the next section by the second-order Boltzmann-Gibbs principle stated in Theorem \ref{t1} below.

\subsection{Second-order Boltzmann-Gibbs principle}
\label{s3.2}
In this section, we state and prove the second-order Boltzmann-Gibbs principle.
Before stating this principle, let us introduce some notation. For a function $v: \bb Z \to \bb R$, let us denote by $\|v\|$ its $\ell^2(\bb Z)$-norm:
\[
\|v\| = \Big\{ \sum_{x \in \bb Z} v(x)^2\Big\}^{1/2}.
\]
Recall the definition of the static compressibility $\chi(\sigma) = \sigma(1-\sigma)$ introduced in Section \ref{s1.7}. For $\ell \in \bb N$ and $\eta \in \Omega$, let us define
\[
\mc Q_\rho(\ell; \eta) = \big(\eta^\ell(0)-\rho\big)^2 - \frac{\chi(\rho)}{\ell}.
\]
We have the following

\begin{theorem}[Second-order Boltzmann-Gibbs principle]
\label{t1}
Let $f: \Omega \to \bb R$ be a local function such that $\varphi_f(\rho) =\varphi'_f(\rho)=0$. There exists a constant $K = K(\rho,f)$ such that for any $\ell \in \bb N$, any $t \geq 0$, any $n \in \bb N$ and any measurable function $v:  [0,T] \times \bb Z \to \bb R$,
\begin{multline}
\label{ec8}
\bb E_n\Big[\Big(\int_0^t \sum_{x \in \bb Z} \tau_x \Big\{f(\eta_s^n) -\frac{\varphi_f''(\rho)}{2}\mc Q_\rho(\ell; \eta_s^n)\Big\} v_s(x) ds\Big)^2\Big] \leq\\
\leq K\Big(\frac{\ell}{n^2}+\frac{t}{\ell^2}\Big) \int_0^t \|v_s\|^2 ds.
\end{multline}
\end{theorem}

\begin{remark}
\label{r3.3}
Given a local function $f: \Omega \to \bb R$, if we choose $\tilde f(\eta) = f(\eta) - \varphi_f(\rho) - \varphi_f'(\rho)(\eta(0)-\rho)$ we can prove that the expectation in \eqref{ec7} tends to 0 as $c(\rho,f,v,t) n^{-1}$, recovering in that way the (first-order) Boltzmann-Gibbs principle stated in Proposition \ref{p7} with an explicit rate of convergence.
\end{remark}

Let us give a careful look at \eqref{ec8}. Theorem \ref{t1} states a second-order correction to Proposition \ref{p7}, with an explicit bound on the variance of the error term. Notice that $\mc Q_\rho(\ell;\eta)$ depends on the configuration of particles on a box of size $\ell$. The natural choice for the size of this box is $\ell = \epsilon n$. In that case, the error term at the right-hand side  of \eqref{ec8} vanishes as $n$ tends to $\infty$ and then $\epsilon \to 0$. Recall that $n$ represents a scaling parameter, being the inverse of the macroscopic distance between neighbors on the lattice. Therefore, a box of size $\epsilon n$ represents an interval of size $\epsilon$ in the scaling limit. We conclude that the function $\mc Q_\rho(\ell; \eta)$ is a {\em quadratic} function of the {\em macroscopic} density of particles around the origin. Therefore, \eqref{ec8} can be interpreted as a {\em replacement lemma} (see Sect.~5 of \cite{KL}) at the level of fluctuations.

The proof of Theorem \ref{t1} follows the classical one-block, two-blocks scheme
introduced by Guo, Papanicolaou, Varadhan in \cite{GPV}. The idea is the following.
In Lemma \ref{l2} below we prove the one-block estimate, which states that, for
some $\ell_0$ large enough, we can replace the function $\tau_x f(\eta_s^n)$ in
\eqref{ec8}, by the conditional expectation $\tau_x \psi_f(\ell_0;
\eta_s^n)$ defined in Proposition \ref{p5}, with an explicit control on the
variance of the error we introduce. Lemma \ref{l3} below is what we call the
renormalization step. It shows that for any $\ell$ large enough, we can replace
the function $\psi_f(\ell;\eta_s^n)$ by $\psi_f(2\ell; \eta_s^n)$, with an error
whose variance is {\em linear} in $\ell$ and $t$. Using $\log \ell$ times Lemma
\ref{l3}, we prove the two-blocks estimate in Lemma \ref{l4}, which explains the
first term on the right-hand side of \eqref{ec8}. The second term on the
right-hand side of \eqref{ec8} is the error obtained when we replace
$\psi_f(\ell;\eta_s^n)$ by $\mc Q_\rho(\ell; \eta_s^n)$ using Proposition \ref{p6}. We start proving the
one-block estimate:

\begin{lemma}[One-block estimate]
\label{l2}
Let $f: \Omega \to \bb R$ be a local function. Assume for simplicity that $\supp(f) \subseteq \{1,...,\ell_0\}$ for some $\ell_0 \in \bb N$. There exists a constant $K_0= K_0(f,\rho)$ such that for any $t \in [0,T]$, any $n \in \bb N$ and any measurable function $v: [0,T] \times \bb Z \to \bb R$ we have
\begin{equation}
\label{ec1.l2}
\bb E_n\Big[\Big(\int_0^t \sum_{x \in \bb Z} \tau_x \big(f(\eta_s^n) - \psi_f(\ell_0; \eta_s^n) \big) v_s(x) ds \Big)^2\Big]
  \leq \frac{K_0 \ell_0^3}{n^2} \int_0^t \|v_s\|^2 ds.
\end{equation}
\end{lemma}

\begin{proof}
Notice that the support of the function $\tau_x(f -\psi_f(\ell_0))$ is equal to the interval $\{x+1,...,x+\ell_0\}$. Therefore, splitting the sum in \eqref{ec1.l2} into $\ell_0$ parts and using Cauchy-Schwarz inequality we obtain expressions for which we can use Corollary \ref{c1}, so that
\begin{align*}
\bb E_n\Big[\Big(\int_0^t
    &\sum_{x \in \bb Z} \tau_x \big(f(\eta_s^n) - \psi_f(\ell_0; \eta_s^n) \big) v_s(x) ds \Big)^2\Big] \leq\\
    &\leq \ell_0 \sum_{i=1}^{\ell_0} \bb E_n\Big[\Big(\int_0^t \sum_{x \in \bb Z} \tau_{\ell_0 x+i} \big(f(\eta_s^n) - \psi_f(\ell_0; \eta_s^n) \big) v_s(\ell_0x+i) ds \Big)^2\Big].\\
\end{align*}
Each term on the right-hand side of this inequality satisfies the hypothesis of Corollary \ref{c1}. Therefore, the left-hand side of \eqref{ec1.l2} is bounded by
\[
\frac{c_2 \ell_0^3}{n^2} \sum_{i=1}^{\ell_0}\sum_{x \in \bb Z} \int_0^t v_s(\ell_0 x+i)^2 ds,
\]
which proves the lemma.
\end{proof}

The assumption $\supp(f) \subseteq \{1,...,\ell_0\}$ in Lemma \ref{l2} is harmless. In fact, by translation invariance, we could consider $\tau_x f$ instead of $f$ for some $x \in \bb Z$, in such a way that the support of $\tau_x f$ is now contained in $\bb N$, from where the existence of $\ell_0$ is granted.
Our aim is to obtain an inequality like \eqref{ec1.l2} with a {\em linear} dependence on the size of the box $\ell$. Notice that the one-block estimate does not use the hypothesis $\varphi_f(\rho)=\varphi_f'(\rho)=0$. The idea is now to continue with the function $\psi_f(\ell)$, which is better behaved than $f$. In particular, and this is the key point in our estimates, by item ii) of Proposition \ref{p6} the variance of $\psi_f(\ell)$ decays like $\ell^{-2}$. This is the content of the following lemma.

\begin{lemma}[Renormalization step]
\label{l3} Let $f: \Omega \to \bb R$ be a local function. Assume for simplicity that there exists $\ell_0 \in \bb N$ such that $\supp(f) \subseteq \{1,...,\ell_0\}$. There exists a constant $K_1=K_1(f,\rho)$ such that for any $t \in [0,T]$, any $n \in \bb N$, any $\ell \geq \ell_0$ and any measurable function $v: [0,T] \times \bb Z \to \bb R$ we have
\begin{equation}
\label{ec1.l3}
\bb E_n\Big[\Big(\int_0^t \sum_{x \in \bb Z} \tau_x \big(\psi_f(\ell; \eta_s^n) -\psi_f(2\ell; \eta_s^n)\big) v_s(x) ds \Big)^2\Big]
  \leq \frac{K_1 \ell^\beta}{n^2} \int_0^t \|v_s\|^2 ds,
\end{equation}
where
\[
\beta=
\left\{
\begin{array}{rl}
2, & \text{if } \varphi'_f(\rho) \neq 0,\\
1, & \text{if } \varphi'_f(\rho) = 0 \text{ and } \varphi''_f(\rho) \neq 0,\\
0, & \text{if } \varphi'_f(\rho) = \varphi_f''(\rho)=0.\\
\end{array}
\right.
\]
\end{lemma}

\begin{proof}
The proof of this lemma is basically the same of Lemma \ref{l2}. The support of the function $\tau_x(\psi_f(\ell)- \psi_f(2\ell))$ is equal to the interval $\{x+1,\dots,x+2\ell\}$. Therefore, the left-hand side of \eqref{ec1.l3} is bounded above by
\begin{equation}
\label{ec2.l3}
2\ell \sum_{i=1}^{2\ell} \bb E_n\Big[ \Big(\int_0^t \sum_{x \in \bb Z} \tau_{2\ell x+i}\big(\psi_f(\ell; \eta_s^n) -\psi_f(2\ell; \eta_s^n)\big)v_s(2\ell x+i)ds \Big)^2\Big].
\end{equation}
Each term in the sum above satisfies the hypothesis of Corollary \ref{c1}. By Proposition \ref{p6},
\[
\int \big(\psi_f(\ell) -\psi_f(2\ell)\big)^2 d\nu_\rho \leq c \ell^{\beta-3}.
\]
Therefore, the sum \eqref{ec2.l3} is bounded by
\[
\frac{2c_2c \ell^\beta}{n^2} \sum_{i=1}^{2\ell} \sum_{x \in \bb Z} \int_0^t
v_s(2 \ell x +i)^2 ds,
\]
which proves the lemma.
\end{proof}

This lemma shows how to benefit from the fact that the function $\psi_f(\ell)$ is much smoother than the original function $f$. Looking at the formulation of this lemma, it is clear that the intention is to double the size of the box (that is, the support of $\psi_f(\ell)$), from the initial size $\ell_0$ until we get a box of the desired size. This is the content of the following lemma:

\begin{lemma}[Two-blocks estimate]
\label{l4}
Let $f: \Omega \to \bb R$ be a local function. Assume for simplicity that there exists $\ell_0 \in \bb N$ such that $\supp(f) \subseteq \{1,...,\ell_0\}$. There exists a constant $K_2=K_2(f,\rho)$ such that for any $t \in [0,T]$, any $n \in \bb N$, any $\ell \geq \ell_0$ and any measurable function $v: [0,T] \times \bb Z \to \bb R$ we have
\begin{equation}
\label{ec1.l4}
\bb E_n\Big[\Big(\int_0^t \sum_{x \in \bb Z} \tau_x \big(\psi_f(\ell_0; \eta_s^n) - \psi_f(\ell; \eta_s^n) \big) v_s(x) ds \Big)^2\Big]
  \leq \frac{K_2 \beta_\ell}{n^2} \int_0^t \|v_s\|^2 ds,
\end{equation}
where
\[
\beta_\ell=
\left\{
\begin{array}{cl}
\ell^2 & \text{if } \varphi'_f(\rho) \neq 0,\\
\ell & \text{if } \varphi'_f(\rho) = 0 \text{ and } \varphi''_f(\rho) \neq 0,\\
(\log \ell)^2 & \text{if } \varphi'_f(\rho) = \varphi_f''(\rho)=0.\\
\end{array}
\right.
\]
\end{lemma}

\begin{proof}
We start proving the lemma for $\ell$ of the form $2^m\ell_0$ for some $m \in \bb N$. We write
\[
\psi_f(\ell_0;\eta_s^n) -\psi_f(\ell;\eta_s^n) = \sum_{i=1}^m \Big\{ \psi_f(2^{i-1}\ell_0;\eta_s^n)-\psi_f(2^i\ell_0;\eta_s^n)\Big\}
\]
and we use Minkowski's inequality to show that the left-hand side of \eqref{ec1.l4} is bounded by
\[
\Big\{\sum_{i=1}^m \bb E_n\Big[\Big(\int_0^t \sum_{x \in \bb Z} \tau_x\big(\psi_f(2^{i-1}\ell_0;\eta_s^n)-\psi_f(2^i\ell_0;\eta_s^n)\big)v_s(x)ds\Big)^2\Big]^{1/2}\Big\}^2.
\]
By Lemma \ref{l3}, this quantity is bounded above by
\[
\Big\{\sum_{i=1}^m\Big(\frac{K_1  2^i\ell_0}{n^2}\int_0^t \|v_s\|^2 ds\Big)^{1/2}\Big\}^2 \leq
\frac{K_2 2^{m\beta}\ell_0^\beta}{n^2}\int_0^t \|v_s\|^2 ds
\]
for some constant $K_2=K_2(f,\rho)$, which proves the lemma for $\ell$ of the form $2^i\ell_0$. For $\ell$ not of this form, we simply choose $m \in \bb N$ such that $2^m \ell_0 <\ell < 2^{m+1}\ell_0$. Using an {\em ad-hoc} version of Lemma \ref{l3}, we can estimate the variance of
\[
\int_0^t \sum_{x \in \bb Z} \tau_x \big(\psi_f(2^m\ell_0; \eta_s^n) -\psi_f(\ell; \eta_s^n)\big) v_s(x) ds
\]
and the lemma follows.
\end{proof}
Finally we replace the function $\psi_\ell(f)$ by a function of the density of particles:

\begin{lemma}
\label{l5}
Let $f: \Omega \to \bb R$ be a local function such that $\varphi_f(\rho)= \varphi_f'(\rho)=0$. Assume for simplicity that there exists $\ell_0 \in \bb N$ such that $\supp(f) \subseteq \{1,...,\ell_0\}$. There exists a constant $K_3=K_3(f,\rho)$ such that for any $t \in [0,T]$, any $n \in \bb N$, any $\ell \geq \ell_0$ and any measurable function $v: [0,T] \times \bb Z \to \bb R$ we have
\[
\bb E_n\Big[\Big(\int_0^t \sum_{x \in \bb Z} \tau_x\Big\{\psi_f(\ell; \eta_s^n)- \frac{\varphi''_f(\rho)}{2} \mc Q_\rho(\ell; \eta_s^n)\Big\}v_s(x) ds\Big)^2\Big] \leq \frac{K_3 t}{\ell^2} \int_0^t \|v_s\|^2 ds.
\]
\end{lemma}
\begin{proof}
The proof of this lemma is fairly simple. We just need to use Cauchy-Schwarz inequality twice. Initially, in order to write the expectation above as $\ell$ sums of functions with supports contained in disjoint intervals (at a multiplicative cost $\ell$). And secondly, to pass the expectation inside the integral (at a multiplicative cost $t$).
\end{proof}

Now the proof of Theorem \ref{t1} follows very easily. We split the expectation
in \eqref{ec8} into three parts, in such a way that each of them can be
estimated using each one of the Lemmas \ref{l2}, \ref{l4}, \ref{l5}. The part
estimated by Lemma \ref{l2} is of lower order than the part estimated by Lemma
\ref{l4}. The two factors in the right-hand side of \eqref{ec8} come exactly
from the estimates in Lemmas \ref{l4} and \ref{l5}.

Stopping the computations right before using Lemma \ref{l5}, we obtain the following result, which will be useful when proving tightness of some associated processes:

\begin{corollary}
\label{c2}
Let $f: \Omega \to \bb R$ be a local function such that $\varphi_f(\rho) = \varphi_f'(\rho)=0$. There exists a constant $K_4=K_4(f,\rho)$ such that for any $t \in [0,T]$, any $n \in \bb N$, any $\ell \in \bb N$ and any measurable function $v: [0,T] \times \bb Z \to \bb R$ we have
\[
\bb E_n\Big[\Big(\int_0^t \sum_{x \in \bb Z} \tau_x \big(f(\eta_s^n)- \psi_f(\ell; \eta_s^n)\big)v_s(x) ds \Big)^2\Big] \leq \frac{K_4 \ell}{n^2} \int_0^ t\|v_s\|^2ds.
\]
\end{corollary}

\section{Proof of Theorem \ref{t0}}
\label{s4}
In this section we prove Theorem \ref{t0}. The main ingredient of the proof is the second-order Boltzman-Gibbs principle stated in Theorem \ref{t1} and proved in the previous section. The proof follows the usual scheme to prove weak convergence theorems of stochastic processes. In Section \ref{s4.2} we prove that the sequence of processes $\{\mc Y_t^n; t \in [0,T]\}_{n \in \bb N}$ is tight with respect to the $J_1$-Skorohod topology of $\mc D([0,T]; \mc S'(\bb R))$. In Section \ref{s4.3} we prove that any limit point of that sequence is an energy solution of the stochastic Burgers equation \eqref{SBE}. In Section \ref{s4.1} we start introducing some martingales associated to the density field $\{\mc Y_t^n; t \in [0,T]\}$.

\subsection{The associated martingales}
\label{s4.1}
Let $g: \Omega \to \bb R$ be a local function. A version of Dynkin's formula tells us that the process $g(\eta_t^n)-g(\eta_0^n)-\int_0^t L_n g(\eta_s^n) ds$ is a martingale. The quadratic variation of this martingale can also be computed and it is equal to $\int_0^t \big\{ L_n g(\eta_s^n)^2-2g(\eta_s^n)L_n g(\eta_s^n)\big\}ds$. We will use this formula for $g(\eta_t^n)= \mc Y_t^n(u)$, where $u$ is a smooth function of compact support. By a limiting procedure, it is easy to see that the formula also holds for $u \in \mc S(\bb R)$. Therefore, fix $u \in \mc S(\bb R)$. For $x \in \bb Z$ and $n \in \bb N$ we define
\[                                                                                                                                                                                                                                                                                                                                                                                                                                                                                                                            \Delta_x^n u = n^2 \big\{u(\tfrac{x+1}{n})+u(\tfrac{x-1}{n})-2u(\tfrac{x}{n})\big\},
\]
\[
\nabla_x^n u=n \big\{u(\tfrac{x+1}{n})-u(\tfrac{x}{n})\big\}.                                                                                                                                                                                                                                                                                                                                                                                                                                                                                                                                                                                                                                         \]
Let $f:\Omega \to \bb R$ be defined as $f(\eta)= ar(\eta)\eta(1)(1-\eta(0))-aD(\rho)\chi(\rho)$. Aside from a scaling factor, this function corresponds to the (centered) asymmetric part of the instantaneous current $j_0^n$ between sites $0$ and $1$. Therefore, $\varphi_f(\sigma) = H(\sigma)-H(\rho)$, where $H(\sigma)$ is the flux defined in Section \ref{s1.7}. According to the convention $H'(\rho)=0$ adopted in Section \ref{s1.7}, we have that $\varphi_f(\rho)=\varphi'_f(\rho)=0$, and $\varphi''_f(\rho)=H''(\rho)$. In particular, the local function $f$ satisfies the hypothesis of the second-order Boltzmann-Gibbs principle stated in Theorem \ref{t1}. Let us define the auxiliary fields
\begin{equation}
\label{def:i}
\mc I_t^n(u) = \int_0^t \frac{1}{\sqrt n} \sum_{x \in \bb Z} \big(\tau_x \omega(\eta_s^n)-\varphi_\omega(\rho)\big)\Delta_x^n u ds,
\end{equation}
\begin{equation}
\label{def:b}
\mc B_t^n(u) = \int_0^t \sum_{x \in \bb Z} \tau_x f(\eta_s^n) \nabla_x^nu ds,
\end{equation}
where $\omega$ was defined in Section \ref{s1.2}.
Applying Dynkin's formula to $\mc Y_t^n(u)$, we see that the process
\begin{equation}
\label{decomp}
\mc M_t^n(u) = \mc Y_t^n(u)- \mc Y_0^n(u) - \mc I_t^n(u) - \mc B_t^n(u)
\end{equation}
is a martingale of quadratic variation
\[
\<\mc M_t^n(u)\> = \int_0^t \frac{1}{n} \sum_{x \in \bb Z} \tau_x \zeta_n(\eta_s^n) \big( \nabla_x^n u\big)^2 ds,
\]
where $\zeta_n(\eta)= r(\eta)\big\{\eta(0)(1-\eta(1))+(1+\frac{a}{\sqrt n})\eta(1)(1-\eta(0))\big\}$.

\subsection{Tightness of the density fluctuation field}
\label{s4.2}
In this section we prove tightness of the sequence of processes $\{\mc Y_t^n; t \in [0,T]\}_{n \in \bb N}$ with respect to the $J_1$-Skorohod topology of $\mc D([0,T]; \mc S'\!(\bb R))$. Moreover, we also prove that any limit point is concentrated on continuous trajectories. The proof will make use of three well-known criteria for tightness of stochastic processes. The first one, due to Mitoma, allows to reduce the proof of tightness of distribution-valued processes to the proof of tightness of real-valued processes.

\begin{proposition}[Mitoma's criterion \cite{Mit}]
\label{Mitoma}
The sequence of $\mc S'(\bb R)$-valued processes $\{\mc Y_t^n; t \in [0,T]\}_{n \in \bb N}$ with trajectories in $\mc D([0,T]; \mc S'(\bb R))$ is tight with respect to the $J_1$-Skorohod topology if and only if the sequence $\{\mc Y_t^n(u); t \in [0,T]\}_{n \in \bb N}$ of real-valued processes is tight with respect to the $J_1$-Skorohod topology of $\mc D([0,T]; \bb R)$ for any $u \in \mc S(\bb R)$.
\end{proposition}

The following criterion, due to Aldous, is very useful in order to check tightness of stochastic processes with respect to the $J_1$-Skorohod topology.

\begin{proposition}
\label{Aldous}
A sequence $\{X_t^n; t \in [0,T]\}_{n \in \bb N}$ of real-valued processes is tight with respect to the $J_1$-Skorohod topology of $\mc D([0,T]; \bb R)$ if:
\begin{itemize}
\item[i)] the sequence of real-valued random variables $\{X_t^n\}_{n \in \bb N}$ is tight for any $t \in [0,T]$,
\item[ii)] for any $\epsilon >0$,
\[
\lim_{\delta \to 0} \limsup_{n \to \infty} \sup_{\gamma \leq \delta} \sup_{\tau \in \mc T_T} \bb P_n(|X_{\tau+\gamma}^n-X_\tau^n| >\epsilon)=0,
\]
where $\mc T_T$ is the set of stopping times bounded by $T$ and where we use the convention $X_{\tau+\gamma}^n=X_T$ if $\tau+\gamma>T$.
\end{itemize}
\end{proposition}

In the case of processes with continuous paths, the following tightness criterion is very effective.

\begin{proposition}[Prohorov-Kolmogorov-Centsov]
\label{PKC}
A sequence $\{X_t^n; t\! \in [0,T]\}_{n \in \bb N}$ of real-valued processes with continuous trajectories is tight with respect to the uniform topology of $\mc C([0,T];\bb R)$ if there exist constants $\kappa$, $\gamma_1$, $\gamma_2>0$ such that
\[
\bb E_n\big[\big|X_t^n-X_s^n\big|^{\gamma_1}\big] \leq \kappa |t-s|^{1+\gamma_2}
\]
for any $s,t \in [0,T]$ and any $n \in \bb N$. Moreover, for any $\alpha <\frac{\gamma_2}{\gamma_1}$, any limit point of the sequence $\{X_t^n; t \in [0,T]\}_{n \in \bb N}$ is concentrated on H\"older-continuous paths of index $\alpha$.
\end{proposition}

Let us begin the proof of tightness of the sequence $\{\mc Y_t^n ; t \in [0,T]\}_{n \in \bb N}$.
By Mitoma's criterion, we only need to prove tightness of $\{\mc Y_t^n(u); t \in [0,T]\}_{n \in \bb N}$ for any $u \in \mc S(\bb R)$. In view of the decomposition \eqref{decomp}, it is enough to prove tightness of the three sequences of processes $\{\mc I_t^n(u); t \in [0,T]\}_{n \in \bb N}$, $\{\mc B_t^n(u); t \in [0,T]\}_{n \in \bb N}$ and $\{\mc M_t^n(u); t \in [0,T]\}_{n \in \bb N}$, and of the sequence $\{\mc Y_0^n(u)\}_{n \in \bb N}$ of real-valued random variables. The last one is the simplest. In fact, computing the characteristic function of $\mc Y_0^n(u)$, it is easy to see that $\mc Y_0^n(u)$ converges in distribution to a normal random variable of mean zero and variance $\chi(\rho)\<u,u\>$, which in particular shows tightness of the sequence $\{\mc Y_0^n(u)\}_{n \in \bb N}$.

Now we turn into the tightness of the martingale $\{\mc M_t^n(u); t \in [0,T]\}_{n\in \bb N}$. For any stopping time $\tau \in \mc T_T$ we have
\begin{align*}
\bb P_n\big(\big|\mc M_{\tau+\gamma}^n(u)-\mc M_\tau^n(u)\big|>\epsilon\big)
    & \leq \frac{1}{\epsilon^2} \bb E_n\big[\big(\mc M_{\tau+\gamma}^n(u)-\mc M_\tau^n(u)\big)^2\big]\\
    & \leq \frac{1}{\epsilon^2} \bb E_n\Big[ \int_\tau^{\tau+\gamma} \frac{1}{n} \sum_{x \in \bb Z} \tau_x \zeta_n(\eta_s^n) \big(\nabla_x^n u)^2 ds\Big]\\
    &\leq \frac{\gamma \|\zeta_n\|_\infty}{\epsilon^2} \frac{1}{n} \sum_{x \in \bb Z} \big(\nabla_x^n u\big)^2.
\end{align*}
For notational convenience, let us define
\[
\mc E_n(u) =  \frac{1}{n} \sum_{x \in \bb Z} \big(\nabla_x^n u\big)^2.
\]
Notice that for any function $u \in \mc S(\bb R)$, $\mc E_n(u)$ tends to $\mc E(u)=\<\nabla u,\nabla u\>$ as $n
$ tends to $\infty$. Notice as well that $\sup_n\|\zeta_n\|_\infty<+\infty$. Therefore, the second condition of Aldous' criterion holds for the sequence $\{\mc M_t^n(u); t \in [0,T]\}_{n \in \bb N}$. The bound
\begin{equation}
\label{L2}
\bb E_n\big[\mc M_t^n(u)^2\big] = \bb E_n\big[\<\mc M_t^n(u)\>\big] \leq \|\zeta_n\|_\infty t \mc E_n(u)
\end{equation}
shows that for any fixed time $t \in [0,T]$, the sequence $\{\mc M_t^n(u)\}_{n\in \bb N}$ is uniformly bounded in $L^2(\bb P_n)$, from where the first condition of Aldous' criterion follows. Therefore, tightness of the sequence $\{\mc M_t^n(u); t \in [0,T]\}_{n \in \bb N}$ follows. The following proposition shows that any limit point of the sequence $\{\mc M_t^n(u); t \in [0,T]\}_{n \in \bb N}$ is concentrated on continuous trajectories.

\begin{proposition}
\label{C-tight}
Let $\{X_t^n; t \in [0,T]\}_{n \in \bb N}$ be a sequence of real-valued processes which is tight with respect to the $J_1$-Skorohod topology of $\mc D([0,T]; \bb R)$. If for any $\epsilon >0$,
\[
\lim_{n \to \infty} \bb P_n\big(\sup_{t \in [0,T]} \big|X_t^n-X_{t-}^n|>\epsilon) =0,
\]
then any limit point of the sequence $\{X_t^n; t \in [0,T]\}_{n \in \bb N}$ is concentrated on continuous trajectories.
\end{proposition}

What this proposition is saying is very natural: if the size of the biggest jump
of the process $\{X_t^n; t \in [0,T]\}$ tends to 0 in probability as $n$ tends
to $\infty$, then any limit point of this sequence is continuous. This
proposition follows easily from the following observation: the size of the
biggest jump is a continuous function in $\mc D([0,T]; \bb R)$. The biggest jump
of the process $\{\mc M_t^n(u); t \in [0,T]\}$ is bounded above by
$\frac{1}{\sqrt n} \|u\|_\infty$ and therefore the sequence $\{\mc
M_t^n(u); t \in [0,T]\}_{n \in \bb N}$ satisfies the hypothesis of this
proposition.

We continue proving tightness for the sequence $\{\mc I_t^n(u); t \in [0,T]\}_{n \in \bb N}$. In this case we will use Proposition \ref{PKC}, and therefore the limit points will be automatically supported on continuous trajectories. By condition iv), the function $\omega$ is local. Let us assume that the support of $\omega$ is contained on the interval $\{x_0+1,...,x_0+\ell_0\}$ for some $x \in \bb Z$ and some $\ell_0 \in \bb N$. For ease of notation, we will also assume that the support of the local functions $f$ and $\zeta_n$ defined above are also contained on $\{x_0+1,...,x_0+\ell_0\}$. In that case, the functions $\tau_x \omega$ and $\tau_y \omega$ are orthogonal in $L^2(\bb \nu_\rho)$ for any $x,y \in \bb Z$ such that $|y-x|\geq \ell_0$. Using Cauchy-Schwarz inequality twice, once to put the expectation inside the time integral and another time to split the sum inside the expectation into sums of orthogonal functions, we see that
\[
\bb E_n\big[ \big(\mc I_t^n(u)-\mc I_s^n(u)\big)^2\big]  \leq \frac{c(\omega,\rho) (t-s)^2}{n} \sum_{x \in \bb Z} (\Delta_x^n u)^2.
\]
Here we have used as well the stationarity of the process $\{\eta_t^n; t \in [0,T]\}$. The constant $c(\omega,\rho)$ can be chosen as $\ell_0 \int (\omega(\eta)-\varphi_\omega(\rho))^2 d\nu_\rho$, but its value is not really important here. What is more important is to notice that
\[
\lim_{n \to \infty} \frac{1}{n} \sum_{x \in \bb Z} (\Delta_x^n u)^2 = \int\limits_{\bb R} \big(\Delta u(x)\big)^2 dx,
\]
from where we see that the hypothesis of Proposition \ref{PKC} is satisfied for the sequence of processes $\{\mc I_t^n(u); t \in [0,T]\}_{n \in \bb N}$ for $\gamma_1 =2$, $\gamma_2=1$ and some $\kappa= \kappa(u,\omega,\rho)$. We conclude that the sequence $\{\mc I_t^n(u); t \in [0,T]\}_{n \in \bb N}$ is tight with respect to the uniform topology of $\mc C([0,T]; \bb R)$ and moreover, any limit point has H\"older-continuous trajectories of index $\alpha$ for any $\alpha <1/2$.

Finally, we are only left to prove tightness for the sequence $\{\mc B_t^n(u); t \in [0,T]\}_{n \in \bb N}$. Of course, this is the most relevant term, since it will generate the nonlinear term on \eqref{SBE}. By Corollary \ref{c2}, we have the estimate
\[
\bb E_n\Big[\Big(\mc B_t^n(u) - \int_0^t \sum_{x \in \bb Z} \psi_f(\ell; \eta_s^n) \nabla_x^n u ds\Big)^2\Big] \leq \frac{K_4 t \ell}{n^2} \sum_{x \in \bb Z} (\nabla_x^nu)^2,
\]
valid for any $t \in [0,T]$ and any $\ell \geq \ell_0$. A similar computation  to the one outlined above for $\bb E_n[\mc I_t^n(u)^2]$ shows that
\begin{equation}
\label{cauchy}
\bb E_n\Big[\Big(\int_0^t \sum_{x \in \bb Z} \psi_f(\ell; \eta_s^n) \nabla_x^n u ds\Big)^2\Big] \leq t^2\ell \int \psi_f(\ell)^2 d\nu_\rho \sum_{x \in \bb Z} (\nabla_x^nu)^2.
\end{equation}
By Remark ii) of Proposition \ref{p6}, we see that $\int \psi_f(\ell)^2
d\nu_\rho \leq c \ell^{-2}$, where the constant $c$ depends only on $f$ and
$\rho$. Recall the definition of the energy $\mc E_n(u)$. Combining these
estimates, we conclude that for $\ell \geq \ell_0$,
\[
\bb E_n[\mc B_t^n(u)^2] \leq K_5 \Big\{\frac{t
\ell}{n}+\frac{t^2n}{\ell}\Big\}\mc E_n(u)
\]
for some constant $K_5=K_5(f,\rho)$. Choosing $\ell$ equal to the integer part
of $n \sqrt t$ we conclude that there exists a constant $K=K(f,\rho)$ such that
\[
\bb E_n[\mc B_t^n(u)^2] \leq K(f,\rho) t^{3/2}\mc E_n(u)
\]
for any $t \geq \frac{\ell_0^2}{n^2}$. This restriction on the value of $t$ comes from the fact that the estimates above hold only for $\ell \geq \ell_0$. For small times $t \leq \frac{\ell_0^2}{n^2}$, a crude Cauchy-Schwarz estimate is enough:
\[
\bb E_n[\mc B_t^n(u)^2] \leq t^2 \int \Big(\sum_{x \in \bb Z} \tau_x f(\eta) \nabla_x^n u\Big)^2 d\nu_\rho \leq \ell_0^2 \int f^2 d \nu_\rho t^{3/2} \mc E_n(u).
\]
Since the process $\{\eta_t^n; t \in [0,T]\}$ is stationary, we have shown the estimate
\begin{equation}
\label{energy}
\begin{split}
\bb E_n[(\mc B_t^n(u)-\mc B_s^n(u))^2]
      &\leq K(f,\rho) |t-s|^{3/2} \mc E_n(u) \\
      &\leq K(u,f,\rho) |t-s|^{3/2}.
\end{split}
\end{equation}
The fact that the constant $K(f,\rho)$ in the first line of this estimate does not depend on $u$ will be useful later on. Therefore, the hypothesis of Proposition \ref{PKC} are satisfied with $\kappa = K(u,f,\rho)$, $\gamma_1=2$ and $\gamma_2=\frac{1}{2}$. We conclude that the sequence $\{\mc B_t^n(u); t \in [0,T]\}_{n \in \bb N}$ is tight with respect to the uniform topology of $\mc C([0,T]; \bb R)$. Moreover, any limit point has H\"older-continuous trajectories of index $\alpha$ for any $\alpha<\frac{1}{4}$.	

In this way, we have just proved that the sequence $\{\mc Y_t^n; t \in [0,T]\}_{n \in \bb N}$ is tight with respect to the $J_1$-Skorohod topology of $\mc D([0,T]; \mc S'(\bb R))$. In addition, we have proved that for any $u \in \mc S(\bb R)$, any limit point of the sequence $\{\mc Y_t^n(u); t \in [0,T]\}_{n \in \bb N}$ is concentrated on continuous trajectories. More properties of these limit points will be discussed in the following section.

\subsection{Identification of limit points}
\label{s4.3}
In this section we finish the proof of Theorem \ref{t0}. We have already proved tightness of the sequence
$\{\mc Y_t^n; t \in [0,T]\}_{n \in \bb N}$. We are left to prove that any limit point is a stationary energy solution of \eqref{SBE}. Actually, we proved a slightly stronger result: any of the sequences $\{\mc M_t^n; t \in [0,T]\}_{n \in \bb N}$, $\{\mc I_t^n; t \in [0,T]\}_{n \in \bb N}$, $\{\mc B_t^n; t \in [0,T]\}_{n \in \bb N}$ is tight. Recall as well that the initial distributions $\{\mc Y_0^n\}_{n \in \bb N}$ converge in distribution, as $n$ tends to $\infty$, to a white noise of variance $\chi(\rho)$. Let $n'$ be  a subsequence for which all the processes above have a limit, and let us denote those limits by $\{\mc Y_t; t \in [0,T]\}$, $\{\mc M_t; t \in [0,T]\}$, $\{\mc I_t; t \in [0,T]\}$, $\{\mc B_t; t \in [0,T]\}$ respectively.

The first thing to prove is that $\{\mc Y_t; t \in [0,T]\}$ has continuous trajectories. Recall that we know only that $\{\mc Y_t(u); t \in [0,T]\}$ has continuous trajectories for any $u \in \mc S(\bb R)$. Taking a dense and countable subset of $\mc S(\bb R)$, we see that $\{\mc Y_t; t \in [0,T]\}$ has continuous trajectories with respect to the weak-$\star$ topology of $\mc S'(\bb R)$. But we already know that $\{\mc Y_t; t \in [0,T]\}$ has c\`adl\`ag trajectories with respect to the strong topology. If both limits, weak-$\star$ and strong, exists, they have to coincide. We conclude that $\{\mc Y_t; t\in [0,T]\}$ has indeed continuous trajectories with respect to the strong topology of $\mc S' (\bb R)$.

Recall the definition of stationary energy solutions of \eqref{SBE} stated in Section \ref{s1.5}. Condition {\bf (S)} is readily satisfied, since for any fixed time $t \in [0,T]$, the sequence $\{\mc Y_t^n\}_{n \in \bb N}$ converges in distribution to a white noise of variance $\chi(\rho)$. The key condition to be verified is {\bf (EC2)}. Recall the definition of $\mc Q(\ell; \eta)$ given in Section \ref{s3.2}. Notice that
\footnote{Here and in what follows, we write $\epsilon n$ for both $\epsilon n$ and its integer part.}
\[
n\mc Q(\epsilon n; \eta_s^n) = \mc Y_s^n(i_\epsilon(0))^2-\frac{\chi(\rho)}{\epsilon}.
\]
Therefore, along the subsequence $n'$, this process converges in distribution to the process $\mc Y_s(i_\epsilon(0))^2 - \frac{\chi(\rho)}{\epsilon}$. We claim that
\[
\mc A_{s,t}^\epsilon(u) = \lim_{n' \to \infty} \int_s^t \sum_{x \in \bb Z} \tau_x \mc Q(\epsilon n'; \eta_s^{n'}) \nabla_x^{n'}\!\!u\, ds.
\]
This limit does not follow directly from the convergence of the sequence $\{\mc Y_t^{n'}; t \in [0,T]\}_{n' \in \bb N}$, since $i_\epsilon(x)$ does not belong to $\mc S(\bb R)$. But if we approximate $i_\epsilon(x)$ by properly chosen functions in $\mc S(\bb R)$, it is not difficult to show that the limit is true. An approximation that works is $\frac{1}{\epsilon}(1-\theta(\frac{x}{ \delta}))\theta(\frac{x-\epsilon}{\delta}+1)$, where $\delta \in (0,\epsilon)$ and $\theta$ is the function introduced in Section \ref{s1.6}. Let us choose $v_s(x) = \nabla_x^n u$ and $\ell = \epsilon n$ in Theorem \ref{t1}. Rewriting \eqref{ec8} in terms of $\mc B_t^{n}$, we see that
\begin{multline*}
\bb E_{n} \Big[\Big( \mc B_t^n(u)-\mc B_s^n(u) - \frac{\varphi''_f(\rho)}{2}\int_s^t \sum_{x \in \bb Z} \tau_x \mc Q(\epsilon n; \eta_r^{n}) \nabla_x^{n}u dr\Big)^2\Big] \leq \\
	\leq K \Big((t-s) \epsilon +\frac{(t-s)^2}{\epsilon^2 n}\Big) \mc
E_n(u)
\end{multline*}
for any $n > \frac{\ell_0}{\epsilon}$.
Notice that upper bounds on moments are stable under convergence in distribution. Taking $n \to \infty$ along the subsequence $n'$ we obtain the bound
\begin{equation}
\label{ec8.04}
\bb E\Big[\Big(\mc B_t(u)-\mc B_s(u)-\frac{\varphi''_f(\rho)}{2} \mc
A_{s,t}^\epsilon(u)\Big)^2\Big] \leq K (t-s) \epsilon \mc E(u).
\end{equation}
By the triangle inequality, we conclude that $\{\mc Y_t; t \in [0,T]\}$
satisfies condition {\bf (EC2)} for $\kappa = \frac{16}{\varphi_f''(\rho)^2}K$.
Condition {\bf (EC1)} is more or less classical in the literature (see Lemma 11.3.10 in \cite{KL}). For the sake of completeness, we give here a sketch of proof built upon Theorem \ref{t1}. We want to prove that there exists a constant $\kappa$ not depending on $u$ or $t$ such that
\[
\limsup_{n \to \infty} \bb E_n[\mc I_t^n(u)^2] \leq \kappa t \mc E(u),
\]
from where condition {\bf (EC1)} follows at once. By Remark \ref{r3.3} right after Theorem \ref{t1}, we see that
\begin{equation}
\label{ec8.05}
\lim_{n \to \infty} \bb E_n\Big[\Big( \mc I_t^n(u) - D(\rho)\int_0^t \mc Y_s^n(\Delta u) ds\Big)^2\Big] =0.
\end{equation}
Notice that this limit identifies $\mc I_t(u)$ as $D(\rho) \int_0^t \mc Y_s(u)ds$. For the moment, \eqref{ec8.05} shows that it is enough to prove that
\begin{equation}
\label{ec8.1}
\limsup_{n \to \infty} \bb E_n\Big[\Big(\int_0^t \mc Y_s^n(\Delta u) ds\Big)^2\Big] \leq \kappa t \mc E(u).
\end{equation}
Using the smoothness of $\Delta u$, we can show that it is enough to prove that
\begin{multline}
\label{ec9}
\limsup_{\epsilon \to 0}\limsup_{n \to \infty} \bb E_n\Big[\Big(\int_0^t
{\sqrt n} \sum_{x \in \epsilon n \bb Z} \big(\eta_s^{n,\epsilon n}(x)-\eta_s^{n,\epsilon n}(x+\epsilon n)\big)\nabla u(\tfrac{x}{n})ds\Big)^2
\Big] \leq \\
		\leq \kappa t \mc E(u),
\end{multline}
where we have used the short notation
\[
\eta_s^{n,\epsilon n}(y) = \frac{1}{\epsilon n} \sum_{i=1}^{\epsilon n} \eta_s^n(y+i)
\]
for $y = x, x+\epsilon n$. The expectation in \eqref{ec9} can be estimated using Corollary \ref{c1}:
\begin{multline*}
\bb E_n\Big[\Big(\int_0^t
{\sqrt n} \sum_{x \in \epsilon n \bb Z} \big(\eta_s^{n,\epsilon n}(x)-\eta_s^{n,\epsilon n}(x+\epsilon n)\big)\nabla u(\tfrac{x}{n})ds\Big)^2
\Big]\leq \\
\leq K(\rho, \epsilon_0) \epsilon t \sum_{x \in \bb Z} \nabla u(\epsilon x)^2,
\end{multline*}
which proves \eqref{ec8.1}.

Up to now we have proved that the process $\{\mc Y_t; t \in [0,T]\}$ satisfies conditions {\bf (S)}, {\bf (EC1)} and {\bf (EC2)}. Theorem \ref{t1.1.5} shows that the $\mc S'(\bb R)$-valued process $\{\mc A_t; t \in [0,T]\}$ given by
\[
\mc A_t(u) = \lim_{\epsilon \to 0} \int_0^t \int\limits_{\bb R} \mc Y_s(i_\epsilon(x))^2 \nabla u(x)dx ds
\]
is well-defined. By \eqref{ec8.04}, $\mc B_t = \frac{1}{2}\varphi''_f(\rho)\mc
A_t$ for any $t \in [0,T]$. Moreover, by the definition of the local function
$f$, $\varphi''(\rho) = a D(\rho)$. By \eqref{ec8.05}, $\mc I_t(u) = D(\rho)
\int_0^t \mc Y_s(\Delta u)ds$ for any $u \in \mc S(\bb R)$ and any $t \in
[0,T]$. Taking  $n \to \infty$ along the subsequence $n'$ in \eqref{decomp}, we
conclude that
\[
\mc M_t(u) = \mc Y_t(u) - \mc Y_0(u) - D(\rho) \int_0^t \mc Y_s(\Delta u) ds - \tfrac{1}{2} a H''(\rho) \mc A_t(u)
\]
for any $u \in \mc S(\bb R)$ and any $t \in [0,T]$.
Therefore, in order to prove that $\{\mc Y_t; t \in [0,T]\}$ is an energy solution of the stochastic Burgers equation \eqref{SBE}, the only thing left to prove is that the process $\{\mc M_t(u); t \in [0,T]\}$ is a martingale of quadratic variation $2\chi(\rho) D(\rho) t \mc E(u)$. A sequence of random variables $\{X_n\}_{n \in \bb N}$ is {\em uniformly integrable} if
\[
\lim_{M \to \infty} \limsup_{n \to \infty} \bb E[|X_n|\textbf{1}_{|X_n|>M}]=0.
\]
Let us recall a very simple criterion to prove that a limit process is a martingale.

\begin{proposition}
\label{iu}
Let $\{M_t^n; t \in [0,T]\}_{n \in \bb N}$ be a sequence of martingales converging in distribution to some process $\{M_t; t \in [0,T]\}$ as $n$ tends to $\infty$. If the sequence of random variables $\{M_T^n\}_{n \in \bb N}$ is uniformly integrable, then $\{M_t; t \in [0,T]\}$ is a martingale.
\end{proposition}

A simple criterion for uniform integrability is a uniform $L^p$-bound for some $p>1$. The estimate \eqref{L2} gives a uniform $L^2$-bound for $\{\mc M_T^n(u)\}_{n \in \bb N}$, which proves that the process $\{\mc M_t(u); t \in [0,T]\}$ is a martingale. The variables $\tau_x \zeta_n$, $\tau_y \zeta_n$ are orthogonal as soon as $|y-x| \geq \ell_0$. Therefore, a simple covariance computation shows that $\<\mc M_t^n(u)\>$ converges in $L^2(\bb P_n)$ to $2 \chi(\rho)D(\rho) t \mc E(u)$ as $n$ tends to $\infty$. By Proposition \ref{iu}, we are left to prove that $\{\mc M_T^n(u)^2-\<\mc M_T^n(u)\>\}_{n \in \bb N}$ is uniformly integrable. But this sequence converges to $\mc M_T(u)^2-2\chi(\rho)D(\rho)t \mc E(u)$ only in $L^1(\bb P_n)$, which is not enough to get uniform integrability. The sequence $\{\<\mc M_T^n(u)\>\}_{n\in \bb N}$ is uniformly bounded in $L^2(\bb P_n)$. The sequences $\{\mc Y_T^n(u)\}_{n \in \bb N}$,
$\{\mc Y_0^n(u)\}_{n \in \bb N}$ and $\{\mc I_T^n(u)\}_{n \in \bb N}$ are uniformly bounded in $L^4(\bb P_n)$. Looking at \eqref{decomp}, we see that are left to prove that the sequence $\{\mc B_T^n(u)^2\}_{n \in \bb N}$ is uniformly integrable. But we do not have anything better than an $L^2(\bb P_n)$-bound for $\mc B_T^n(u)$. The idea is the following. It is not true that a sequence which is bounded in $L^1$ is uniformly integrable. But it is true that a sequence converging to 0 in $L^1$ is uniformly integrable. Therefore, if we can take out a piece of $\mc B_T^n(u)$ converging to 0 in $L^2(\bb P_n)$ in such a way that the rest is uniformly bounded in, let us say, $L^4(\bb P_n)$, then we are done. This idea can be formalized by means of the following elementary estimate: let $X_1, X_2$ be two random variables and let us consider $X = X_1+X_2$. Then,
\begin{equation}
\label{estimate}
\begin{split}
\bb E[X^2\textbf{1}_{X^2>M}]
      &\leq 2\bb E[X_1^2\textbf{1}_{X^2>M}]+2\bb E[X_2^2\textbf{1}_{X^2>M}]\\
      &\leq 2\bb E[X_1^2]+2 \sqrt{\bb E[X_2^4]\bb P(X^2>M)}\\
      &\leq 2\bigg(\bb E[X_1^2]+\sqrt{\frac{\bb E[X_2^4]\bb E[X^2]}{M}}\bigg).\\
\end{split}
\end{equation}
We will use this estimate for $X = \mc B_T^n(u)$ and
\[
X_1 = \int_0^T \sum_{x \in \bb Z} \tau_x(f(\eta_s^n)-\psi_f(\ell; \eta_s^n))\nabla_x^nu ds,
\]
\[
X_2 = \int_0^T \sum_{x \in \bb Z} \tau_x\psi_f(\ell; \eta_s^n)\nabla_x^nu ds.
\]
By Corollary \ref{c2}, $\bb E_n[X_1^2] \leq \frac{C\ell}{n}$ for some constant $C=C(u,T,f,\rho)$. By the energy estimate \eqref{energy}, $\bb E_n[X^2] \leq C$ for some other constant $C=C(u,T,f,\rho)$. By a computation similar to the one performed in \eqref{cauchy}, $\bb E_n[X_2^4] \leq C(\frac{n}{\ell}+\frac{n^2}{\ell^2})$ for some constant $C=C(u,T,f,\rho)$. Now we fix $\epsilon \in (0,1)$ and we choose $\ell = \epsilon n$, which is bigger than $\ell_0$ for $n$ big enough. We conclude that
\[
\bb E_n[\mc B_T^n(u)^2 \textbf{1}_{\mc B_T^n(u)^2>M}] \leq C\Big(\epsilon + \frac{1}{\epsilon \sqrt M}\Big).
\]
Sending $n \to  \infty$, then $M \to \infty$ and finally $\epsilon \to 0$, we conclude that the sequence $\{\mc B_T^n(u)^2\}_{n \in \bb N}$ is uniformly integrable, which finishes the proof of Theorem \ref{t0}.

\section{Energy solutions of the stochastic Burgers equation}
\label{s6}

In this section we prove Theorems \ref{t1.1.5}, \ref{t2.1.5}, \ref{t1.1.6} and \ref{t2.1.6}. We start proving Theorem \ref{t1.1.5}.
\subsection{Proof of Theorem \ref{t1.1.5}}
\label{s6.1}
Let $\{\mc Y_t; t \in [0,T]\}$ be a process satisfying conditions {\bf (S)} and {\bf (EC2)}. By condition {\bf (EC2)}, for any $t \in [0,T]$ and $u \in \mc S(\bb R)$ fixed, the limit
\[
\mc A_t(u)= \lim_{\epsilon \to 0} \mc A_{0,t}^\epsilon(u)
\]
exists in $L^2$. Taking $\delta \to 0$ in the energy estimate {\bf (EC2)}, we obtain the estimate
\[
\bb E\big[\big(\mc A^\epsilon_{0,t}(u)- \mc A_t(u)\big)^2\big] \leq \kappa t
\epsilon \mc E(u).
\]
By condition {\bf (S)} and the Cauchy-Schwarz inequality we have the bound $\bb E[\mc A_{0,t}^\epsilon(u)^2] \leq \frac{\chi
t^2}{\epsilon} \mc E(u)$.
Therefore,
\[
\bb E[\mc A_t(u)^2]  \leq 2\Big(\kappa t \epsilon + \frac{\chi
t^2}{\epsilon}\Big)\mc E(u),
\]
where we have used the elementary estimate $(x+y)^2 \leq 2(x^2+y^2)$. Choosing
$\epsilon = \sqrt t$, we obtain the bound $\bb E[\mc A_t(u)^2] \leq
2(\kappa+\chi)t^{3/2} \mc E(u)$. By stationarity, we conclude that the process
$\{\mc A_t(u); t \in [0,T]\}$ satisfies the bound
\[
\bb E\big[\big(\mc A_t(u)-\mc A_s(u)\big)^2\big] \leq 2(\kappa+\chi)|t-s|^{3/2} \mc E(u)
\]
for any $s,t \in [0,T]$. By the Kolmogorov-Centsov Theorem (which is a
particular case of Proposition \ref{PKC}), the process $\{\mc A_t(u); t \in
[0,T]\}$ has a version (which we still denote by $\{\mc A_t(u); t \in [0,T]\}$)
that is H\"older-continuous of index $\alpha$ for any $\alpha<\frac{1}{4}$.
Moreover, this version still satisfies the bound
\begin{equation}
\label{en1}
\bb E[\mc A_t(u)^2] \leq K t^{3/2}\mc E(u)
\end{equation}
for any $t \in [0,T]$ and any $u \in \mc S(\bb R)$. An important aspect on this bound is its dependence on $u$. The functional $u \mapsto \mc E(u)$ is continuous in $\mc S(\bb R)$. It turns out that this is enough to guarantee the existence of a $\mc S'(\bb R)$-valued process $\{\mc A_t; t \in [0,T]\}$ such that $\{\mc A_t(u); t \in [0,T]\}$ is the evaluation of $\{\mc A_t; t \in [0,T]\}$ on the test function $u$, for any $u \in \mc S(\bb R)$. This fact is more or less well-known in the literature, but instead of quoting a proper reference, we explain how to construct such a process using Mitoma's criterion. Recall the definition of the Hermite function $\{\her_k; k \in \bb N\}$. Let $E_k \subseteq \mc S(\bb R)$ be the space generated by $\{\her_1,...,\her_k\}$. Let $\pi_k: L^2(\bb R) \to E_k$ denote the canonical orthogonal projection. For $u \in \mc S(\bb R)$, define $\mc A_{t,k}(u) = \mc A_t(u_k)$, where the function $u_k$ is the only function in $\mc S(\bb R)$ satisfying $\nabla u_k = \pi_k \nabla u$. The processes $\{\mc A_{t,k}; t \in [0,T]\}$ are well-defined, $\mc S'(\bb R)$-valued processes for any $k \in \bb N$. Using \eqref{en1}, we can check that $\mc A_{t,k}(u)$ tends to $\mc A_t(u)$ as $k$ tends to $\infty$, for any $t \in [0,T]$ and any $u \in \mc S(\bb R)$. Moreover, the sequence $\{\mc A_{t,k}; t \in [0,T]\}_{k \in \bb N}$ satisfies the hypothesis of Mitoma's criterion. Therefore, it has a limit point which is an $\mc S'(\bb R)$-valued process. Calling this limit point $\{\mc A_t; t \in [0,T]\}$, we end the proof of Theorem \ref{t1.1.5}.

\subsection{Proof of Theorem \ref{t2.1.5}}
\label{s6.2}
Let us prove Theorem \ref{t2.1.5}. We start by observing that we have already shown that $\{\mc A_t(u); t \in [0,T]\}$ is H\"older-continuous of index $\alpha$ for any $\alpha <\frac{1}{4}$ and any $u \in \mc S(\bb R)$. This is basically all we need to show. Let us recall the martingale decomposition
\begin{equation}
\label{ec6.1}
\mc Y_t(u)=\mc Y_0(u) + D \int_0^t \mc Y_s(\Delta u)ds + \lambda \mc A_t(u) +\sqrt{2\chi D} \mc B_t(\nabla u),
\end{equation}
where $\mc B_t(\nabla u)$ is a continuous martingale of quadratic variation $\mc
E(u)t$. In particular, $\{\mc B_t(\nabla u); t \in [0,T]\}$ is
H\"older-continuous of index $\alpha$ for any $\alpha < \frac{1}{2}$. The bound
\[
\bb E\Big[\Big(\int_0^t \mc Y_s(\Delta u)ds\Big)^2\Big] \leq t^2 \chi
\int\limits_{\bb R} (\Delta u(x))^2 dx
\]
shows that the integral term in \eqref{ec6.1} is H\"older-continuous of index $\alpha$ for any $\alpha <1$, which ends the proof of Theorem \ref{t2.1.5}.

\subsection{Proof of Theorem \ref{t1.1.6}}
\label{s6.3}
Theorem \ref{t1.1.6} follows from the following observation. On one hand, for the weakly asymmetric, simple exclusion process (the model for which $r \equiv 1$), Theorem \ref{t0.1} holds. On the other hand, Bertini and Giacomin showed in \cite{BG} that the height fluctuation field associated to the weakly asymmetric, simple exclusion process converges to the Cole-Hopf solution of the KPZ equation.  Therefore, the Cole-Hopf solution of the KPZ equation, which is a limit point of the height fluctuation field of the weakly asymmetric, simple exclusion process, is an energy solution of the KPZ equation as well, which shows Theorem \ref{t1.1.6}.

\subsection{Proof of Theorem \ref{t2.1.6}}
\label{s6.4}
Let us prove Theorem~\ref{t2.1.6}. Let $\{\mc Y_t; t \in [0,T]\}$ be a stationary energy solution of \eqref{SBE}. Let us define the $\mc S'(\bb R)$-valued process $\{\mc I_t; t\in [0,T]\}$ as
\[
\mc I_t(u) = D \int_0^t \mc Y_s(\Delta u) ds
\]
for any $u \in \mc S(\bb R)$ and any $t \in [0,T]$.
Notice that this notation is in concordance with the one used in Section \ref{s4}. We have the martingale decomposition
\begin{equation}
\label{ec1.s6}
\mc Y_t(u) - \mc Y_0(u) = \mc I_t(u) + \lambda \mc A_t(u) + \mc M_t(u),
\end{equation}
where the process $\{\mc M_t(u); t \in [0,T]\}$ is a continuous martingale of quadratic
variation $2 \chi D t \mc E(u)$. In particular, by {\bf (EC1)} and \eqref{en1},
there is a constant $K>0$ such that
\[
\bb E\big[\big(\mc Y_t(u)-\mc Y_0(u)\big)^2\big] \leq K(t+t^{3/2})\mc E(u)
\]
for any $u \in \mc S(\bb R)$ and any $t \in [0,T]$. Recall the definition of the functions $F$, $F_M$ made in Section~\ref{s1.6}. Notice that
\[
\lim_{M,N \to \infty} \mc E(F_M-F_N) =0.
\]
In fact, this follows directly from \eqref{ec2.5}.
We conclude that the sequence $\{\mc Y_t(F_M)-\mc Y_0(F_M)\}_{M \in \bb N}$ is Cauchy in $L^2$. Let us call $\{\mc Y^\ast_t(F)-\mc Y_0^\ast(F); t \in [0,T]\}$ the corresponding limit process. Notice that we have not proved anything about the path properties of the process $\{\mc Y_t^\ast(F)-\mc Y_0^\ast(F); t \in [0,T]\}$. It can also be verified that $\int_{\bb R} (\Delta(F_M-F_N))^2 dx$ tends to $0$ as $M,N$ tend to $\infty$. Using this fact, we can prove tightness of the sequence $\{\mc Y_t(F_M)-\mc Y_0(F_M)\}_{M \in \bb N}$ with respect to the uniform topology of $\mc C([0,T]; \bb R)$ like we did it for $\{\mc Y_t^n(u); t \in [0,T]\}_{n \in \bb N}$ in Section \ref{s4.2}. This shows that the process $\{\mc Y_t^\ast(F)-\mc Y_0^\ast(F); t \in [0,T]\}$ has continuous trajectories. We are left to prove that the process $\{h_t(x); t \in [0,T], x \in \bb R\}$ defined in \eqref{def_h}
is an almost stationary energy solution of the KPZ equation \eqref{KPZ}.
Recall that for any mean zero function $u \in \mc S(\bb R)$, its primitive $U(x) = \int_{-\infty}^x u(y)dy$ also belongs to $\mc S(\bb R)$. Therefore,
\[
\<h_t,u\> = \mc Y_t(U)-\mc Y_0(U),
\]
and moreover, this relation determines in a unique way the process $\{h_t(x)-h_t(0); t \in [0,T], x \in \bb R\}$. In particular, $\{h_t(x); x \in \bb R\}$ is a two-sided Brownian motion of variance $\chi$, and $\{h_t(x); t \in [0,T], x \in \bb R\}$ satisfies condition \textbf{(S')}. The energy condition for the process $\{h_t(x); t \in [0,T], x \in \bb R\}$ is just the energy condition for the process $\{\mc Y_t;t \in [0,T]\}$, written in terms of the process $\{h_t(x); t \in [0,T], x \in \bb R\}$. Therefore, we are left to prove that for any $u \in \mc S(\bb R)$, the process
\[
\<h_t,u\>-\<h_0,u\> - D\int_0^t \<h_s,\Delta u\> ds -\lambda \tilde{\mc A}_{0,t}(u)
\]
is a continuous martingale of quadratic variation $2 \chi D\<u,u\>t$. By \eqref{ec1.s6}, this is true for any function $u \in \mc S(\bb R)$ such that $\int_{\bb R} u(x) dx =0$. By linearity, we are left to show this for just a single test function $u \in \mc S(\bb R)$ such that $\int u(x) dx \neq 0$. The simplest choice is $u(x) = e^x(1+e^x)^{-2}= \nabla F(x)$. Define $u_M = \nabla F_M$. Notice that $u_M$ tends to $u$ in $L^2(\bb R)$ as $M$ tends to $\infty$. Moreover, since $\int_{\bb R} u_M(x) dx=0$, the process
\[
\<h_t,u_M\>-\<h_0,u_M\> - D\int_0^t \<h_s,\Delta u_M\> ds -\lambda \tilde{\mc A}_{0,t}(u_M)
\]
is a continuous martingale of quadratic variation $2 \chi D\<u_M,u_M\>t$. All the terms above are continuous under approximations in $L^2(\bb R)$, and therefore we conclude that
\[
\<h_t,u\>-\<h_0,u\> - D\int_0^t \<h_s,\Delta u\> ds -\lambda \tilde{\mc A}_{0,t}(u)
\]
is indeed a continuous martingale. We are only left to prove that the quadratic
variation of this martingale is equal to $2\chi D \<u,u\>t$. We proceed like in
Section \ref{s4.3}. The convergence of the quadratic variations is clear, and we
just need to show that the sequence of martingales is uniformly integrable. The
terms $\<h_t,u_M\>$, $\<h_0,u_M\>$ and $D \int_0^t \<h_s,\Delta u_M\> ds$ are
uniformly bounded in $L^4$. The uniform integrability of $\tilde{\mc A}_{0,t}(u_M)$ follows from \eqref{estimate} and the elementary bound
\[
\bb E\big[\tilde{\mc A}_{0,t}^\epsilon(u)^4\big] \leq C\Big(\frac{1}{\epsilon} \int\limits_{\bb R} u(x)^4 dx +\frac{1}{\epsilon^2} \Big(\int\limits_{\bb R} u(x)^2 dx \Big)^2\Big)
\]
for some constant $C$ depending only on $t$ and $\chi$. This ends the proof of Theorem \ref{t2.1.6}.

\section{Proof of Theorem \ref{t0.1}}
\label{s5}

In this section we prove Theorem \ref{t0.1}. The simplest way to do this is using Theorem \ref{t2.1.6}. In fact, the proof basically consists in showing that the estimates made in Section \ref{s6.4} hold uniformly in $n$. The height fluctuation field $\{\mc H_t^n(x); t \in [0,T], x \in \bb R\}$ and the density fluctuation field $\{\mc Y_t^n; t \in [0,T]\}$ are related by the following identity: for any $u \in \mc S(\bb R)$ and any $t \in [0,T]$,
\[
\int\limits_{\bb R} \mc H_t^n(x) \nabla u(x) dx = \mc Y_t^n(u).
\]
In order to simplify the notation, for a function $u \in \mc S(\bb R)$ and $n \in \bb N$, $t \in [0,T]$ we write
\[
\mc H_t^n(u) = \int\limits_{\bb R} \mc H_t^n(x) u(x) dx.
\]

Notice that for any function $u \in \mc S(\bb R)$ such that $\int_{\bb R} u(x) dx=0$, the function $U(x) = \int_{-\infty}^x u(y) dy$ belongs to $\mc S(\bb R)$ and therefore the relation
\[
\mc H_t^n(u)  = \mc Y_t^n(U)
\]
shows tightness of the sequence $\{ \mc H_t^n(u); t \in [0,T]\}_{n \in \bb N}$ with respect to the $J_1$-Skorohod topology of $\mc D([0,T]; \bb R)$. In order to conclude tightness of the sequence $\{\mc H_t^n(x); t \in [0,T], x \in \bb R\}_{n \in \bb N}$, we need to show tightness of the sequence $\{\mc H_t^n(u); t \in [0,T]\}_{n \in \bb N}$ for $u \in \mc S(\bb R)$ such that $\int_{\bb R} u(x) dx \neq 0$. By linearity, we need to prove this for just a single function $u \in \mc S(\bb R)$ such that $\int_{\bb R} u(x) dx=0$.
As in Section \ref{s6.4}, we choose $u(x) = \nabla F(x) = e^x(1+e^x)^{-2}$.
Notice that
\[
\lim_{M,N \to \infty} \sup_{n\in \bb N} \mc E_n(F_M-F_N)=0,
\]
\[
\lim_{M, N \to \infty} \sup_{n \in \bb N} \frac{1}{n} \sum_{x \in \bb Z} \big(\Delta_x^nF_M-\Delta_x^nF_N\big)^2 =0.
\]
This follows from the smoothness of the functions $F$, $\theta$ and the continuous versions of these limits.
Therefore, not only the sequences of random variables $\{\mc Y_t^n(F_M)- \mc Y_0^n(F_M)\}_{M \in \bb N}$ are Cauchy in $L^2$, uniformly in $n$, but the family of processes $\{\mc Y_t^n(F_M)-\mc Y_0^n(F_M); t \in [0,T]\}_{n, M \in \bb N}$ is tight. Let us call $\{\mc Y_t^n(F)-\mc Y_0^n(F); t \in [0,T]\}$ the limit process of $\{\mc Y_t^n(F_M)- \mc Y_t^n(F_M); t \in [0,T]\}_{M \in \bb N}$. At the microscopic level, we see that
\begin{align*}
\mc H_t^n(\nabla F)
	&= \mc H_0^n(\nabla F) +\lim_{M \to \infty} \big(\mc H_t^n(\nabla F_M)-\mc H_0^n(\nabla F_M)\big) \\
	&= \mc H_0^n(\nabla F) +\lim_{M \to \infty} \big(\mc Y_t^n( F_M)-\mc Y_0^n(F_M)\big)\\
	&= \mc H_0^n(\nabla F) + \mc Y_t^n(F)-\mc Y_0^n(F).
\end{align*}
Tightness of $\{\mc H_t^n(\nabla F); t \in [0,T]\}_{n \in \bb N}$ follows from tightness of $\{\mc Y_t^n(F_M)-\mc Y_0^n(F_M); t \in [0,T]\}_{n,M \in \bb N}$. Therefore, by Mitoma's criterion the sequence $\{\mc H_t^n(x); t \in [0,T], x \in \bb N\}_{n \in \bb N}$ is tight. Taking an adequate subsequence, we see that for any limit point $\{\mc H_t(x);t \in [0,T],  x\in \bb N\}$ of the sequence $\{\mc H_t^n(x); t \in [0,T], x \in \bb N\}_{n \in \bb N}$, there is a stationary energy solution $\{\mc Y_t; t \in [0,T]\}$ of \eqref{SBE} such that
\[
\mc H_t(u) = \mc Y_t(U-\bar{u} \nabla F)-\mc Y_0(U-\bar{u} \nabla F)-\mc H_0(u)+\bar{u}\big(\mc Y_t^*(F)-\mc Y^*_0(F)\big)
\]
for any $u \in \mc S(\bb R)$. By Theorem \ref{t2.1.6}, we conclude that $\{\mc H_t(x); t \in [0,T], x \in \bb R\}$ is an almost stationary, energy solution of the KPZ equation \eqref{KPZ}.

\vspace{0.5cm}

\noindent{\bf Acknowledgments.}\\

P.G. thanks the warm hospitality of the ``Universit\'e
Paris-Dauphine'', where this work was initiated, of IMPA and of
``Courant Institute of Mathematical Sciences'', where this work was
finished.

P.G. thanks FCT (Portugal) for support through the research
project ``Non-Equilibrium Statistical Physics" PTDC/MAT/109844/2009.
P.G. also thanks the Research Centre of Mathematics of the University of
Minho, for the financial support provided by ``FEDER" through the
``Programa Operacional Factores de Competitividade  COMPETE" and by
FCT through the research project PEst-C/MAT/UI0013/2011.

M.J. would like to thank the warm hospitality of the ``Fields Institute''. M.J. acknowledges support of CNPq, through the project 479514/2011-9 Universal and of FAPERJ, through the project JCNE E-26.103.051/2012.

Both authors thank FCT (Portugal) and Capes (Brazil) for the financial support through the
research project ``Non-Equilibrium Statistical Mechanics of Stochastic
Lattice Systems".


\end{document}